\documentclass[preprint,11pt]{elsarticle}
\usepackage{amsmath,amssymb,amsthm,ascmac,ifdraft,xcolor,url,booktabs,natbib}
\usepackage{algorithm,algpseudocode,enumerate,graphicx,wrapfig,caption,subcaption,comment}
\newcommand{\rev}[2][red]{\textcolor{#1}{#2}}

\DeclareMathOperator{\Fix}{\mathrm{Fix}}
\DeclareMathOperator{\Id}{\mathrm{Id}}
\DeclareMathOperator{\argmin}{argmin}
\DeclareMathOperator{\dom}{dom}
\DeclareMathOperator{\lev}{lev}
\DeclareMathOperator{\bd}{bd}
\DeclareMathOperator{\cl}{cl}
\DeclareMathOperator{\diam}{diam}
\newcommand{\nat}{\mathbb{N}}
\newcommand{\real}{\mathbb{R}}
\newcommand{\Ball}{\mathbf{B}}
\newcommand{\Sphere}{\mathbf{S}}
\newcommand{\ip}[1]{\mathchoice{\left\langle#1\right\rangle}{\langle#1\rangle}{\langle#1\rangle}{\langle#1\rangle}}
\newcommand{\abs}[1]{\mathchoice{\left\lvert#1\right\rvert}{\lvert#1\rvert}{\lvert#1\rvert}{\lvert#1\rvert}}
\newcommand{\norm}[1]{\mathchoice{\left\lVert#1\right\rVert}{\lVert#1\rVert}{\lVert#1\rVert}{\lVert#1\rVert}}
\newcommand{\Hsp}{H}

\theoremstyle{definition}
\newtheorem{problem}{Problem}[section]

\newtheorem{assum}{Assumption}[section]
\newtheorem{dfn}{Definition}[section]
\theoremstyle{plain}
\newtheorem{example}{Example}[section]
\newtheorem{prop}{Proposition}[section]
\newtheorem{lem}{Lemma}[section]
\newtheorem{cor}{Corollary}[section]
\newtheorem{thm}{Theorem}[section]

\newcommand{\GitHubURL}{\url{https://github.com/iiduka-researches/201811-kaz}}

\newcounter{assumptions}

\newif\ifextended\extendedtrue

\title{Fixed Point Quasiconvex Subgradient Method}

\newcommand{\postaladdress}{1-1-1 Higashimita, Tama-ku, Kawasaki-shi, Kanagawa 214-8571 Japan}
\author[kaz]{Kazuhiro Hishinuma\footnote{Research Fellow of the Japan Society for the Promotion of Science.}}
\address[kaz]{Computer Science Course, Graduate School of Science and Technology, Meiji University, \postaladdress.}
\ead{kaz@cs.meiji.ac.jp}
\author[iiduka]{Hideaki Iiduka}
\address[iiduka]{Department of Computer Science, Meiji University, \postaladdress.}
\ead{iiduka@cs.meiji.ac.jp}

\begin{document}
\begin{abstract}
Constrained quasiconvex optimization problems appear in many fields, such as economics, engineering, and management science.
In particular, fractional programming, which models ratio indicators such as the profit/cost ratio as fractional objective functions, is an important instance.
Subgradient methods and their variants are useful ways for solving these problems efficiently.
Many complicated constraint sets onto which it is hard to compute the metric projections in a realistic amount of time appear in these applications.
This implies that the existing methods cannot be applied to quasiconvex optimization over a complicated set.
Meanwhile, thanks to fixed point theory, we can construct a computable nonexpansive mapping whose fixed point set coincides with a complicated constraint set.
This paper proposes an algorithm that uses a computable nonexpansive mapping for solving a constrained quasiconvex optimization problem.
We provide convergence analyses for constant diminishing step-size rules.
Numerical comparisons between the proposed algorithm and an existing algorithm show that the proposed algorithm runs stably and quickly even when the running time of the existing algorithm exceeds the time limit.
\end{abstract}
\maketitle

%%%%%%%%%%%%%%%%%%%%%%%%%%%%%%%%%%%%%%%%%%%%%%%%%%%%%%%%%%%%%%%%%%%%%%%%%%%%%%%
%%% INTRODUCTION
%%%%%%%%%%%%%%%%%%%%%%%%%%%%%%%%%%%%%%%%%%%%%%%%%%%%%%%%%%%%%%%%%%%%%%%%%%%%%%%
\section{Introduction}
%% 主問題について (この論文で解決する問題)
This paper considers the constrained quasiconvex optimization problem.
This problem is composed of a quasiconvex objective functional and a closed convex constraint set.
We call a functional of which any slice is convex a quasiconvex functional, and the class of this functional is a generalization of convex functionals.
Quasiconvex functionals inherit some nice properties of convex functionals \cite{yaohua2016}.
However, they do not have all the important properties of convex functionals, such as convexity of the sum of convex functionals, or give a guarantee of the coincidence of local optimality and global optimality.
Therefore, the constrained quasiconvex optimization problem is difficult to solve in general.

%% 主問題に対する動機 (なぜこの問題を解決しないといけないのか)
\begin{comment}
The development of a useful algorithm for solving the constrained quasiconvex optimization problem has three important purposes.
First, it would contribute to solving constrained convex optimization problems.
Such problems are actively studied \cite{iiduka2015,iiduka2016,iiduka2016oms,iiduka2016ejor} and have various applications, such as signal recovery \cite{combettes2007}, machine learning \cite{fujiwara2018,hayashi2018}, and network resource allocation \cite{iiduka2013,iiduka2014}.
Because any convex functional is also a quasiconvex functional, we can deal with these applications in the framework of quasiconvex optimization.
This implies that studying quasiconvex optimization will help to develop the field of convex optimization and expand its applications.
\end{comment}

\label{M:13}Fractional programming is an important instance of constrained quasiconvex optimization problems.
In economics, there are various situations in which one optimizes ratio indicators, such as the debt/equity ratio (in financial and corporate planning), inventory/sales and output/employee ratios (in production planning), and cost/patient and nurse/patient ratios (in health care and hospital planning) \cite{stancu1997}.
Under certain conditions, these ratio indicators, fractional objective functionals in other words, have quasiconvexity \cite[Lemma~3]{kiwiel2001}.
Therefore, these problems can be dealt with as constrained quasiconvex optimizations.
Here, we will examine the numerical behaviors of the existing and proposed algorithms when they are applied to the Cobb-Douglas production efficiency problem \cite[Problem~(3.13)]{bradley1974}, \cite[Problem~(6.1)]{yaohua2015}, \cite[Section~1.7]{stancu1997}, which is an instance of a fractional programming and constrained quasiconvex optimization problem.
Furthermore, the demand for techniques to solve optimization problems is nowadays not only limited to convex objectives.
In particular, optimization problems whose objective functionals are quasiconvex have appeared in economics, engineering, and management science \cite{yaohua2015,yaohua2016}.
Therefore, this paper builds an algorithm that can efficiently solve constrained quasiconvex optimization problems even if they have some complexity.

%% これまでの研究論文について考察
Subgradient methods with the usual Fenchel subdifferential, an expansion of the gradient for nonsmooth functionals, are useful for solving problems in convex optimization \cite[Section~8.2]{bertsekas2003}, \cite{iiduka2015,iiduka2016,iiduka2016oms,iiduka2016ejor,larsson1996}.
\label{M:6}We need to use an alternative notion of the usual Fenchel subdifferential since the usual Fenchel subdifferential is defined for a convex functional \cite[Subsection~2.1]{censor2006}, \cite[Subsection~2.2]{kiwiel2001}, \cite[Proposition~8.12]{rockafellar}.
Indeed, the usual Fenchel subdifferential may be empty even for a differential nonconvex functional \cite[Subsection~2.1]{censor2006}.
This implies that we cannot use it directly to solve quasiconvex optimization problems.
\label{M2:1}Fortunately, various, extended subdifferentials for nonconvex functionals have been proposed \cite[Section~4]{penot1998}, \cite[Definition~8.3]{rockafellar}.
As an instance of them, we can define subdifferentials for quasiconvex or more general functionals by the procedure described in \cite[Section~4]{penot1998} to construct them with directional derivatives.
These subdifferentials inherit some of the properties, called axioms for subdifferentials \cite[Axioms~($S_1$--$S_4$)]{penot1998}, from the usual Fenchel subdifferential for convex functionals; however, they may not be easily computable.
Furthermore, an essential issue is that a local minimizer might not coincide with the global minimizer in quasiconvex optimization.
This issue reduces the subdifferential till it contains only one vector, i.e., the zero vector, meaning that the methods lose any clue as to the direction of the global minimizer.
Hence, we cannot ensure the convergence of the generated sequence to the global minimizer of the quasiconvex optimization problem when the usual subgradient methods are used.
For the unconstrained quasiconvex optimization problem, Konnov \cite{konnov2003} introduced a subgradient method that uses a normalized normal vector to the slice at a current approximation as a subgradient.
This idea overcomes the above issue, since there certainly exists a nonzero normal vector to the slice which indicates the direction to the minimizer even if the current approximation is a non-global local minimizer.

Kiwiel \cite{kiwiel2001} proposed a subgradient method that uses a normalized normal vector to the slice as a subgradient (we will call it a subgradient throughout this paper) for solving the constrained quasiconvex optimization problem.
Hu et al. \cite{yaohua2015} analyzed its convergence properties when inexact subgradients are used and/or when it includes computational errors.
Furthermore, a number of subgradient-method variants exist for solving quasiconvex optimization problems, such as the conditional subgradient methods \cite{yaohua2016} and the stochastic subgradient method \cite{yaohua2016b}.

The existing methods assume the computability of the metric projection onto the constraint set, because they use the metric projection to guarantee that the solution is in the constraint set.
The metric projection onto the constraint set is defined as a mapping which translates a given point into the nearest point inside the constraint set.
Therefore, in general, we have to solve a subproblem of minimizing the distance from a given point subject to the solution being in the constraint set.
Certainly, there are some sets onto which the metric projections can be computed easily, such as boxes \cite[Proposition~29.15]{bc}, closed balls \cite[Example~3.18 and Proposition~3.19]{bc}, \cite[Section~4]{sakurai2014}, and closed half-spaces \cite[Example~29.20]{bc}.
However, various complicated sets on which computing the metric projections is difficult appear in practical problems \cite{combettes1999,iiduka2013,iiduka2016ejor,iiduka2017,yamada2001}.
Therefore, we have to develop a new algorithm that can run lightly and quickly even when it is difficult to compute the metric projection onto the constraint set.

On the other hand, if the constraint set can be expressed as a fixed point set of or the intersection of some fixed point sets of nonexpansive mapping(s), there are algorithms that use these nonexpansive mappings instead of the metric projection for convex optimization \cite{iiduka2015,iiduka2016,iiduka2016oms,iiduka2016ejor}.
Fixed point sets of nonexpansive mappings have great powers of expression.
Any metric projection onto a closed convex set is also a nonexpansive mapping whose fixed point set coincides with these sets \cite[Proposition~4.16]{bc}.
We can build a nonexpansive mapping whose fixed point set coincides with the intersection of the fixed point sets of two or more given nonexpansive mappings \cite[Proposition~4.9, 4.47]{bc}.
Furthermore, there are complicated convex sets called generalized convex feasible sets that are defined by closed convex sets whose intersection may be empty.
They can also be expressed using concrete nonexpansive mappings \cite[Definition~(8)]{iiduka2016ejor}, \cite[Definition~(50)]{yamada2001}.
The algorithms listed at the beginning of this paragraph use nonexpansive mappings instead of metric projections onto the constraint sets.
Therefore, if these nonexpansive mappings can be more easily computed than the metric projections, it can also be expected that their algorithms will run more efficiently than algorithms which use metric projections directly.

The existing algorithms for solving convex optimization problems over fixed point sets of nonexpansive mappings are realized by combining a fixed point iterator, which generates a sequence converging to some fixed point of a given nonexpansive mapping, with the existing subgradient methods.
The Krasnosel'skii-Mann iterator \cite{krasnoselskii1955,mann1953} and Halpern iterator \cite{halpern1967} are useful fixed point iterators for finding a fixed point of given nonexpansive mapping.
Both generate a sequence converging to some fixed point of a given nonexpansive mapping.

%% 論文の立ち位置 (この論文ではなにをするか)
In contrast to the existing literature, this paper proposes a novel algorithm which minimizes a given quasiconvex functional over the fixed point set of a given nonexpansive mapping.
To realize this algorithm, we combine the Krasnosel'skii-Mann iterator \cite{krasnoselskii1955,mann1953} with the existing subgradient method \cite{kiwiel2001} for solving quasiconvex optimization problems.
The goal of this paper is to show that our algorithm can solve constrained quasiconvex optimization problems whose constraint set is too complex for the existing algorithms to solve in a realistic amount of time.

%% 論文の貢献について (既存研究に対して)
This paper offers three contributions.
The first is to provide a widely applicable algorithm for solving constrained quasiconvex optimization problems.
The nonexpansive mappings are an extended notion of the metric projection, since the metric projection is also nonexpansive.
Therefore, this paper allows more varied modeling for constrained quasiconvex optimization problems.

The second contribution is to present the theoretical convergence properties of our algorithm.
We analyzed the convergence properties for constant and diminishing step-size rules.
These results show by how much the error increases when a constant step-size rule is adopted what conditions are required for the generated sequence to converge to the solution of the optimization problem.

The last contribution is to overcome the issue of the existing methods; that is, we show that the proposed algorithm can solve problems whose metric projections onto constraint sets cannot be easily computed.
We conduct a numerical comparison of our algorithm and the existing algorithm.
The results show that our algorithm can solve actual problems even when the constraint sets are too complex to find the metric projection onto them and when the existing algorithm cannot run in a realistic amount of time.

%% ２章では…，３章では…，４章では…．
This paper is organized as follows.
Section~\ref{sec:preliminaries} gives the mathematical preliminaries.
Section~\ref{sec:convergence} defines our algorithm and presents its convergence analyses.
Section~\ref{sec:experiment} shows numerical comparisons between the proposed algorithm and the existing subgradient method, by solving a constrained quasiconvex optimization problem named the Cobb-Douglas production efficiency problem.
Section~\ref{sec:conclusion} concludes this paper. % and mentions future directions for considering wider applications.
The Appendices include miscellaneous propositions, lemmas, and their proofs.

%%%%%%%%%%%%%%%%%%%%%%%%%%%%%%%%%%%%%%%%%%%%%%%%%%%%%%%%%%%%%%%%%%%%%%%%%%%%%%%
%%% PRELIMINARIES
%%%%%%%%%%%%%%%%%%%%%%%%%%%%%%%%%%%%%%%%%%%%%%%%%%%%%%%%%%%%%%%%%%%%%%%%%%%%%%%
\section{Mathematical Preliminaries}\label{sec:preliminaries}
First, we present the main problem considered in this paper, i.e., Problem~\ref{prob:main}, which is called a \emph{constrained quasiconvex optimization problem}.
\begin{problem}\label{prob:main}
Let $\Hsp$ be a real Hilbert space with inner product $\ip{\cdot,\cdot}$ and its induced norm $\norm{\cdot}$, and let $f$ be a continuous functional on $\Hsp$.
In addition, suppose that the functional $f$ has \emph{quasiconvexity}, i.e. $f((1-\alpha)x+\alpha y)\le\max\{f(x), f(y)\}$ holds for any $x,y\in\Hsp$ and for any $\alpha\in[0,1]$.
Let $X,D$ be nonempty closed convex subsets of $\Hsp$.
Then, we would like to
\label{M:8}\begin{align*}
\text{minimize }f(x)
\text{ subject to }x\in X\cap D.
\end{align*}
\end{problem}
We define the \emph{set of minima} and the \emph{minimum value} of Problem~\ref{prob:main} by $X^\star:=\argmin_{x\in X\cap D}f(x)$ and $f_\star:=\inf_{x\in X\cap D}f(x)$, respectively.

We use the following notation in this paper.
$\nat$ is the set of natural numbers without zero, and $\real$ is the set of real numbers.
$\Ball:=\{x\in\Hsp:\norm{x}\le 1\}$ is the unit ball in Hilbert space, and $\Sphere:=\{x\in\Hsp:\norm{x}=1\}$ is the unit sphere in that space.
$\Id$ is the identity mapping of $\Hsp$ onto itself.
The boundary of a set $C\subset\Hsp$ is denoted by $\bd C$, and the closure of this set is denoted by $\cl C$.
The metric projection onto a closed, convex set $C\subset\Hsp$ is denoted by $P_C$ and defined as $P_C(x)\in C$ and $\norm{x-P_C(x)}=\inf_{y\in C}\norm{x-y}$ for any $x\in\Hsp$.
For any $\alpha\in\real$, the $\alpha$-slice of a functional $f:\Hsp\to\real$ is denoted by $\lev_{<\alpha}f:=\{x\in\Hsp:f(x)<\alpha\}$.
The fixed point set of a mapping $T:\Hsp\to\Hsp$ is denoted by $\Fix(T):=\{x\in\Hsp:T(x)=x\}$.

This paper makes full use of the nonexpansivity of some mapping for analyzing the convergence of the proposed algorithm.
Hence, let us define two kinds of nonexpansive condition.
A mapping $T:\Hsp\to\Hsp$ is said to be nonexpansive if $\norm{T(x)-T(y)}\le\norm{x-y}$ for any $x,y\in\Hsp$, and it is said to be firmly nonexpansive if $\norm{T(x)-T(y)}^2+\norm{(\Id-T)x-(\Id-T)y}^2\le\norm{x-y}^2$ for any $x,y\in\Hsp$.
Obviously, a firmly nonexpansive mapping is also a nonexpansive mapping \cite[Subchapter~4.1]{bc}. The properties of these nonexpansivities are described in detail in \cite[Chapter~4]{bc}, \cite[Chapter~6]{wataru}.

Useful algorithms for solving Problem~\ref{prob:main} were proposed in \cite{yaohua2015,yaohua2016,kiwiel2001,konnov2003}.
However, they assume that the metric projection onto the set $X$ can be computed explicitly.
Unfortunately, there are many complicated convex sets onto which constructing and/or computing the projection are difficult \cite{combettes1999,iiduka2016ejor,iiduka2017,yamada2001}.
This paper assumes a weaker and expanded condition for the constraint set $X$, only requiring the existence of a certain nonexpansive mapping expressing this set.
Below, we list the conditions assumed throughout in this paper.
\begin{assum}\label{assum:basic}
We suppose that
\begin{enumerate}[({A}1)]
\item
the effective domain $\dom(f):=\{x\in\Hsp:f(x)<\infty\}$ coincides with the whole space $\Hsp$;
\item\label{a2:fixedpoint}
there exists some firmly nonexpansive mapping $T:\Hsp\to\Hsp$ whose fixed point set $\Fix(T)$ coincides with the constraint set $X$;
\item\label{a3:nonempty}
the constraint set $X=\Fix(T)$ and the feasible set $X\cap D$ are nonempty and there exists at least one minima, i.e. $X^\star\neq\emptyset$.
\setcounter{assumptions}{\value{enumi}}
\end{enumerate}
\end{assum}

Assumptions (A\ref{a2:fixedpoint}--\ref{a3:nonempty}) mean that any closed convex sets which can be expressed as a fixed point set of some (firmly) nonexpansive mapping are accepted as constraint sets.
Fixed point sets of nonexpansive mappings can express a variety of constraint sets, including not only the sets onto which the metric projections can be calculated such as is used in the existing literature \cite{yaohua2015,yaohua2016,kiwiel2001}, but also complicated sets onto which metric projections cannot be easily calculated \cite{combettes1999,iiduka2016ejor,iiduka2017,yamada2001}.

We can construct more complex sets by combining simpler nonexpansive mappings.
The following proposition gives the fundamental, variously applicable transformations for building nonexpansive mappings.
\begin{prop}\label{prop:trans}
Let $T_1, T_2, \ldots, T_N:\Hsp\to\Hsp$ be nonexpansive mappings (including metric projections onto some convex sets), and suppose that the intersection of these fixed point sets is nonempty.
Let $P_C$ be a metric projection onto a nonempty, closed, convex set $C\subset\Hsp$, and assume that $C\cap\Fix(T_1)\neq\emptyset$.
Let $\alpha\in(0,1/2]$.
Then,
\begin{enumerate}[\rm ({T}1)]
\item\label{t:average}
the mapping $\sum_{i=1}^NT_i/N$ is also a nonexpansive mapping, and its fixed point set coincides with $\bigcap_{i=1}^N\Fix(T_i)$ \cite[Propositions~4.9 and 4.47]{bc};
\item\label{t:firmup}
the mapping $\alpha\Id+(1-\alpha)T_1$ is firmly nonexpansive mapping, and its fixed point set coincides with $\Fix(T_1)$ \cite[Remark~4.37 and Proposition~4.47]{bc}.
\end{enumerate}
\end{prop}
The transformation~(T\ref{t:average}) ensures that we can make a nonexpansive mapping whose fixed point set coincides with the intersection of the fixed point sets of two or more nonexpansive mappings.
The transformation~(T\ref{t:firmup}) provides us with a way to convert any nonexpansive mapping into a firmly nonexpansive mapping whose fixed point sets correspond with the given one.
Our GitHub repository (URL: \GitHubURL{}) provides these implementations of the transformations (T\ref{t:average} and T\ref{t:firmup}) as higher-order functions \texttt{average} and \texttt{firm\_up} in Python.
By using our code, the reader can easily make a nonexpansive mapping expressing his or her desired constraint set.

Furthermore, let us examine an instance of convex sets that can be expressed as fixed point sets of some nonexpansive mappings, called the \emph{generalized convex feasible sets} \cite[Definition~(10)]{iiduka2016ejor}, \cite[Subsection~4.B]{yamada2001}.
Here, let us consider several closed convex sets $X_i\subset\Hsp$ for $i=0,1,\ldots,K$, and suppose that the metric projections $\{P_{X_i}\}_{i=0}^K$ onto these convex sets $\{X_i\}_{i=0}^K$ can be easily calculated.
If the intersection of these sets is not empty, we can use the transformation and construction procedures described before to make a nonexpansive mapping whose fixed point set coincides with it.
Hence, let us consider the opposite case; that is, there is a possibility that the intersection of the sets $\{X_i\}_{i=1}^K$ is empty.
Then, we cannot use the straightforward way because the emptiness of the constraint set violates Assumption~(A\ref{a3:nonempty}).
To design an alternative constraint set, let us define a functional~\cite[Definition~(8)]{iiduka2016ejor}, \cite[Definition~(50)]{yamada2001}
\begin{align}
g(x):=\frac{1}{2}\sum_{i=1}^Kw_i\left(\min_{y\in X_i}\norm{x-y}\right)^2\quad(x\in\Hsp), \label{dfn:gfcs}
\end{align}
where $\sum_{i=1}^Kw_i=1$. This functional $g$ stands for the mean square value from the point $x$ to the sets $\{X_i\}_{i=1}^K$ with respect to the weights $\{w_i\}_{i=1}^K$. Therefore, we can consider the set of points which minimize this functional \cite[Definition~(10)]{iiduka2016ejor}, \cite[Definition~(50)]{yamada2001},
\begin{align*}
X_g:=\left\{x\in X_0:g(x)=\min_{y\in X_0}g(y)\right\},
\end{align*}
as an alternative constraint set in terms of the mean square norm.
This set is called the generalized convex feasible set.
We can construct a nonexpansive mapping whose fixed point set coincides with this set, and thus, we can deal with the minimization problem over this constraint set by using the algorithm presented later.
The way to construct this nonexpansive mapping is described in \cite[Definition~(9)]{iiduka2016ejor}, \cite[Definition~(50)]{yamada2001}.

\label{M2:2}Various subdifferentials have been proposed for quasiconvex functionals, such as the classical subdifferential \cite[Definition~(10)]{kiwiel2001}, the Greenberg-Pierskalla subdifferential \cite[Section~3]{greenberg1973} and its variants such as the star subdifferential \cite[Definition~7]{censor2006}, Plastria's lower subdifferential \cite[Section~2]{plastria1985}, and so on \cite[Subsection~2.1]{censor2006}, \cite[Definition~(6)--(9)]{kiwiel2001}, \cite[Section~5]{penot1998}.
The Greenberg-Pierskalla subdifferential is one of the most important concepts of subdifferentials for generalized convex functionals because it is a general notion that can be easily handled \cite[Section~5]{penot1998}.
However, it does not account for the norm of its subgradients and gives only directions.
Plastria's lower subdifferential is proposed as another important concept whose properties are closer to those of the usual Fenchel subdifferential for convex functionals \cite[Section~5]{penot1998}.
In this paper, conforming to \cite{yaohua2015,yaohua2016,konnov2003}, we use the \emph{subdifferential} defined as the normal cones to the slice of the functional $f$.
That is, given a point $x\in\Hsp$, we call the set
\begin{align*}
\partial^\star f(x):=\{g\in\Hsp:\ip{g,y-x}\le 0\ (y\in\lev_{<f(x)}f)\}
\end{align*}
the subdifferential of the quasiconvex functional $f$ at a point $x\in\Hsp$ \cite[Definition~2.3]{yaohua2015}, \cite[Definition~2.1]{yaohua2016}, \cite[Definition~(9)]{kiwiel2001}, \cite[Section~1]{konnov2003}.
We also call its element a \emph{subgradient}.

This subdifferential $\partial^\star f(\cdot)$ has some favorable properties, as listed in the following proposition.
\begin{prop}[{\cite[Proposition~6.2.4]{bc}, \cite[Lemma~3]{kiwiel2001}, \cite[Propositions~6 and 8]{penot1998}}]\label{M:7}
  Suppose that Assumption~\ref{assum:basic} holds, and assume that $f$ is a continuous quasiconvex functional.
  Then, the following hold.
  \begin{enumerate}[\rm({P}1)]
    \item\label{p2} $\partial^\star f(\cdot)$ coincides with the closure of the Greenberg-Pierskalla subdifferential, i.e., the union of the Greenberg-Pierskalla subdifferential and the singleton set $\{0\}$.
    \item\label{p4} The Greenberg-Pierskalla subdifferential, Plastria's lower subdifferential, and the usual Fenchel subdifferential are contained in the subdifferential $\partial^\star f(x)$ for any $x\in\Hsp$.\footnote{Furthermore, the other four kinds of subdifferential presented in \cite[Section~5]{penot1998} are also contained in the subdifferential $\partial^\star f(x)$. Please refer to \cite[Proposition~6]{penot1998} for more details.}
    \item\label{p1} For any $x\in\Hsp$, the subdifferential $\partial^\star f(x)$ is nonempty, and also contains some nonzero vector.
    \item\label{p3} $\partial^\star f(\cdot)$ is a nonempty closed convex cone.
  \end{enumerate}
\end{prop}
First, we defined the subdifferential $\partial^\star f(\cdot)$ as the closure of the Greenberg-Pierskalla subdifferential, such as is shown in Proposition~(P\ref{p2}).
This implies that the subdifferential $\partial^\star f(\cdot)$ is an extension of the Greenberg-Pierskalla subdifferential, and some properties of this subdifferential can also be used.
For example, $\partial^\star f(x)$ coincides with the whole space $\Hsp$ if $x$ is a minimizer of $f$.
This proposition ensures that the subdifferential $\partial^\star f(\cdot)$ coincides with the closure of the Greenberg-Pierskalla subdifferential, which is not always closed \cite[Subsection~2.1]{yaohua2015}.
Hence, this subdifferential $\partial^\star f(\cdot)$ overcomes the problem of the non-closedness of the Greenberg-Pierskalla subdifferential; it has been used in the recent literature \cite{yaohua2015,yaohua2016b,yaohua2016}.
Furthermore, as shown in Proposition (P\ref{p4}), the subdifferential is also a superset of Plastria's lower subdifferential and the usual Fenchel subdifferential.
Since the subdifferential $\partial^\star f(\cdot)$ is a cone as shown in Proposition (P\ref{p3}), this property ensures that every arbitrarily scaled element of the Plastria's lower subdifferential or the usual Fenchel subdifferential can be used as a subgradient in the discussion of this paper.
Proposition~(P\ref{p1}) ensures the existence of nonzero subgradients at all points.
This fact guarantees that the algorithm described later can always find a subgradient, which is required for the computation.
In addition, Proposition~(P\ref{p3}) shows that the normalized vector of a subgradient is also a subgradient.
Our algorithm implicitly uses this property for choosing a subgradient whose norm is 1.

A subgradient in $\partial^\star f(\cdot)$ is computable when, for example, the functional is formed as a fractional function, a typical instance of a quasiconvex function, with concrete conditions.
The following proposition gives the conditions for the quasiconvexity and subgradient computability of fractional functions.
\begin{prop}[{\cite[Lemma~3 (i), 4]{kiwiel2001}}]\label{prop:fraction}
Let $a$ be a convex functional on $\Hsp$, and let $b$ be a finite, positive functional on $\Hsp$.
Suppose that $f(x):=a(x)/b(x)$ for any $x\in\Hsp$, the interior of $\dom(f)$ is convex, and one of the following conditions holds:
\begin{enumerate}[(i)]
\item $b$ is affine;
\item $a$ is nonnegative on the interior of $\dom(f)$ and $b$ is concave;
\item $a$ is nonpositive on the interior of $\dom(f)$ and $b$ is convex.
\end{enumerate}
Then, the functional $f$ is a quasiconvex functional on the interior of $\dom(f)$.
Furthermore, the functional $(a-\alpha b)(\cdot)$ is convex and $\partial(a-\alpha b)(x)\subset\partial^\star f(x)$ for any $x\in\Hsp$, where $\alpha:=f(x)$ and $\partial(a-\alpha b)$ is the usual Fenchel subdifferential of the functional $(a-\alpha b)(\cdot)$.
\end{prop}

The following defines a property named the H\"older condition of a functional.
This property is used in turn to describe some of the properties of the quasi-subgradient and to establish the convergence of subgradient methods \cite[Section~2]{yaohua2016}.
\begin{dfn}[H\"older condition~{\cite[Definition~1]{konnov2003}}]
A functional $f:\Hsp\to\real$ is said to satisfy the \emph{H\"older condition} with degree $\beta>0$ at a point $x\in\Hsp$ on a set $M\subset\Hsp$ if there exists a number $L\in\real$ such that
\begin{align*}
\abs{f(z)-f(x)}
\le L\norm{z-x}^\beta\quad(z\in M).
\end{align*}
\end{dfn}
The H\"older condition with degree 1 is equivalent to Lipschitz continuity.
Furthermore, when $f$ is a convex functional, it is also equivalent to the bounded subgradient assumption frequently assumed in convergence analyses of subgradient methods for solving convex optimization problems \cite[Section~2]{yaohua2016}.
For more details on this property, see Example~\ref{example:1} described later.

The following Proposition~\ref{prop:konnov} is a key lemma which relates the distance to the set of minima to its functional value.
While nearly the same assertion in Euclidian spaces is presented in \cite[Proposition~2.1]{konnov2003}, this proposition extends it to Hilbert spaces.
We should remark that the condition for a point $x$ is slightly modified from the original one for the later discussion.
Nevertheless, we can similarly prove this proposition.
\begin{prop}[{\cite[Proposition~2.1]{konnov2003}}]\label{prop:konnov}
Suppose that the functional $f$ satisfies the H\"older condition with degree $\beta>0$ at a point $x^\star\in X^\star$ on the set $\cl(\lev_{<f(x)}f)$ for some point $x\in\Hsp$ such that $f_\star<f(x)$.
Then, we have
\begin{align*}
f(x)-f_\star
\le L\ip{g,x-x^\star}^\beta\quad(g\in\partial^\star f(x)\cap\Sphere).
\end{align*}
\end{prop}

The following propositions are used to prove the theorems presented later.

\begin{prop}[{\cite[Lemma~1]{opial1967}}]\label{prop:opial}
Let $\{x_k\}$ be a sequence in the Hilbert space $\Hsp$ and suppose that it converges weakly to $x$.
Then for any $y\neq x$, $\liminf_{k\to\infty}\norm{x_k-x}<\liminf_{k\to\infty}\norm{x_k-y}$.
\end{prop}

\begin{prop}[{\cite[Proposition~10.25]{bc}}]\label{prop:weaklycontinuous}
Every quasiconvex continuous functional on a real Hilbert space $\Hsp$ has weakly lower semicontinuity.
That is to say, we have $f(x)\le\liminf_{n\to\infty}f(x_n)$ for any sequence $\{x_n\}\subset\Hsp$ which converges weakly to a point $x\in\Hsp$ if the functional $f$ is a quasiconvex continuous functional on a real Hilbert space $\Hsp$.
\end{prop}

\begin{prop}[{\cite[Proposition~11.8]{bc}}]\label{prop:uniqueness}
Let $C$ be a convex subset of $\Hsp$, and let $f$ be a strictly quasiconvex functional on the Hilbert space $\Hsp$, i.e., $f(\alpha x+(1-\alpha)y)<\max\{f(x),f(y)\}$ for any $\alpha\in(0,1)$ and for any two distinct points $x,y\in\Hsp$.
Then, $f$ has at most one minimizer over $C$.
\end{prop}

%%%%%%%%%%%%%%%%%%%%%%%%%%%%%%%%%%%%%%%%%%%%%%%%%%%%%%%%%%%%%%%%%%%%%%%%%%%%%%%
%%% CONVERGENCE ANALYSIS
%%%%%%%%%%%%%%%%%%%%%%%%%%%%%%%%%%%%%%%%%%%%%%%%%%%%%%%%%%%%%%%%%%%%%%%%%%%%%%%
\section{Quasiconvex Subgradient Method over a Fixed Point Set}\label{sec:convergence}
We propose the following Algorithm~\ref{alg:main} for solving Problem~\ref{prob:main} with Assumption~\ref{assum:basic}.
\begin{algorithm}
  \caption{Fixed Point Quasiconvex Subgradient Method for Solving Problem~\ref{prob:main}}\label{alg:main}
  \begin{algorithmic}[1]
    \Require
      \Statex{$f$: $\Hsp\to\real$, $T$: $\Hsp\to \Hsp$, $D\subset\Hsp$;}
      \Statex{$\{v_k\}\subset(0,\infty)$, $\{\alpha_k\}\subset(0,1]$.}
    \Ensure
      \Statex{$\{x_k\}\subset D$.}
    \State{$x_1\in D$.}
    \For{$k=1,2,\ldots$}
      \State{$g_k\in\partial^\star f(x_k)\cap\Sphere$.}\label{alg:main:1}
      \State{$x_{k+1}:=P_D(\alpha_kx_k+(1-\alpha_k)T(x_k-v_kg_k))$.}\label{alg:main:improve}
    \EndFor
  \end{algorithmic}
\end{algorithm}
This algorithm iteratively generates the next point $x_{k+1}$ from the current approximation $x_k$ in order to improve it.
Specifically, step~\ref{alg:main:1} of this algorithm finds a regularized subgradient of the functional $f$ at the current approximation $x_k$.
Step~\ref{alg:main:improve} is composed of two improving iterators: one is the subgradient method iterator $x_k-v_kg_k$ to improve approximations with respect to the functional value, and the other is the Krasnosel'ski\u\i-Mann iterator \cite{krasnoselskii1955,mann1953} $\alpha_k\Id+(1-\alpha_k)T$ to improve approximations with respect to the distance to the fixed point set $\Fix(T)$.
\label{M:9}To ensure that the generated sequence is contained in the set $D$, we project each generated point onto this set (this operation is optional because the metric projection operator $P_D$ coincides with the identity mapping $\Id$ if the set $D$ is the whole space $\Hsp$).
By repeating steps~\ref{alg:main:1}--\ref{alg:main:improve}, this algorithm generates a sequence converging to a point in the solution set $X^\star$.

\begin{comment}
\label{pg:assum}
Assumption~\ref{assum:basic} supposes that the effective domain $\dom(f)$ coincides with the whole space $\Hsp$.
Nevertheless, we can also apply to this algorithm a functional whose effective domain does not fill $\Hsp$.
For example, let us consider a case that the effective domain $\dom(f)$ differs from the whole space, the range of $T$ is convex and contained inside this domain, and the initial point $x_1$ is an element of this range. In step~\ref{alg:main:improve} of Algorithm~\ref{alg:main}, $x_{k+1}$ is defined as a convex combination of $x_k$ and $T(x_k-v_kg_k)$; both are obviously members of the range of $T$.
Since we assumed that the range of $T$ is convex, the generated point $x_{k+1}$ is also an element of this range.
Hence, we can consider the computation of Algorithm~\ref{alg:main} is limited in the range of $T$, a subset of the effective domain $\dom(f)$.
This implies that, with appropriate approximation of $f$, we can apply a functional whose effective domain does not fill the whole space of this algorithm when the range of the nonexpansive mapping $T$ is contained in this domain. (We will illustrate an example of making an approximate function as Example~\ref{example:3}.)
\end{comment}

Before moving on to the convergence analyses, we will give the assumptions and lemmas describing the fundamental properties of Algorithm~\ref{alg:main}.
\begin{assum}\label{assum:0}
\begin{enumerate}[({A}1)]
\setcounter{enumi}{\value{assumptions}}
\item\label{assum:holder}
For any $k\in\nat$ such that $f_\star<f(x_k)$ and for all $x^\star\in X^\star$, the functional $f$ satisfies the H\"older condition with degree $\beta>0$ at the point $x^\star$ on the set $\cl(\lev_{<f(x_k)}f)$.
\item\label{assum:bounded}
The generated sequence $\{x_k\}$ is bounded.
\item\label{assum:alpha}
The real sequence $\{\alpha_k\}\subset(0,1]$ satisfies $0<\liminf_{k\to\infty}\alpha_k\le\limsup_{k\to\infty}\alpha_k<1$.
\setcounter{assumptions}{\value{enumi}}
\end{enumerate}
\end{assum}
\label{M:10}In the following, we have to ensure that Assumption (A\ref{assum:bounded}), i.e., the boundedness of the sequence generated by Algorithm~\ref{alg:main}, holds.
If we know the estimated range of the solution candidates, the simplest way to bound the generated sequence is to let the set $D$ be a ball with a large enough diameter.
We can compute the metric projection onto a closed ball easily \cite[Example~3.18]{bc}.
For example, giving $10^{16}\Ball$ as the set $D$ to Algorithm 1 satisfies Assumption~(A\ref{assum:bounded}).
\label{M2:3}\rev{Even when the boundedness of the set $D$ cannot be guaranteed, we can ensure the boundedness of the generated sequence if the objective functional is coercive.
A detailed discussion and proof of this fact is given in {\ifextended\ref{appendix:loft}\else Appendix~A\fi}.}

Here let us give some examples which satisfy Assumption~\ref{assum:0}.
The first example shows the applicability of Algorithm~\ref{alg:main} to constrained convex optimization problems.
\begin{example}[{\cite[Remark~4.34.(iii), Propositions~4.47 and 16.20]{bc}}]\label{example:1}
Suppose that $f$ is a continuous, convex functional on $\Hsp$, $\tilde{T}$ is a nonexpansive mapping of $\Hsp$ into itself, and $D$ is a closed, bounded, convex subset of $\Hsp$.
Set $T:=(\Id+\tilde{T})/2$.
Furthermore, assume that the feasible set $\Fix(T)\cap D$ is nonempty.
Set $\alpha_k:=1/2$ for every $k\in\nat$.
If $\Hsp$ is finite-dimensional, or if the usual Fenchel subdifferential of $f$ maps every bounded subset of H to a bounded set, then $T$ is a firmly nonexpansive mapping, $\Fix(T)$ coincides with the intersection $\Fix(T)$, and Assumption~\ref{assum:0} holds.
\end{example}
This example shows that, if the set $D$ is bounded, our algorithm can be applied to nonsmooth convex optimization problems over fixed point sets of nonexpansive mappings \cite{iiduka2015,iiduka2016,iiduka2016oms,iiduka2016ejor}.
\label{M:11}The following discussion is predicated upon Assumption~(A\ref{assum:bounded}), i.e., that the generated sequence is bounded, since it is required for evaluating the distance between the generated sequence and the fixed point set (Lemma~\ref{lem:dist}).
Indeed, the existing analysis of the fixed point subgradient method for convex optimization assumes the generated sequence is bounded \cite[Assumption~(A2)]{Iiduka2015b}, \cite[Assumption~3.1]{iiduka2015}.
However, we can guarantee the boundedness. As we described above, the simplest $D$ that satisfies the requirements is a ball with a large enough diameter.

\label{M:1}The transformation $\tilde{T}\mapsto(\Id+\tilde{T})/2$ converts the given nonexpansive mapping $\tilde{T}$ into a corresponding firmly nonexpansive mapping $T$ whose fixed point set coincides with the given one \cite[Remark~4.37 and Proposition~4.47]{bc}.
This implies that any nonexpansive mapping whose fixed point set coincides with the constraint set can be used as $\tilde{T}$.
Of course, since any metric projection operator is a firmly nonexpansive mapping \cite[Proposition~4.16]{bc}, we can solve an optimization problem over the constraint set onto which the metric projection can be computed.
In Section~\ref{sec:experiment}, we will describe a concrete example of constructing a firmly nonexpansive mapping.
Furthermore, our algorithm extends the existing subgradient methods for convex optimization, since any of the usual Fenchel subgradients of a convex functional is also a subgradient as defined in this paper.
Hence, the convergence analyses of our algorithm (described later) will be very useful for not only quasiconvex optimization \cite{yaohua2015,yaohua2016,kiwiel2001,konnov2003} but also convex optimization \cite{iiduka2015,iiduka2016,iiduka2016oms,iiduka2016ejor}.

\label{M:12}Let us consider the simplest example of a (nonconvex) quasiconvex objective functional.
The next example shows that a typical quasiconvex functional, called the capped-$l_1$ norm, appearing in sparse regularization of machine learning tasks \cite[Equation~(25)]{cheung2017}, \cite[Appendix~C.3.1]{zhang2010} can be minimized using Algorithm~\ref{alg:main}.
\begin{example}\label{example:2}
Let $f(x):=\min\{\norm{x},\alpha\}$ for some $\alpha>0$, and let $T:=\Id$.
Set $\{v_k\}\subset(0,\alpha]$ and $\alpha_k:=1/2$ for all $k\in\nat$.
Use $g_k:=x_k/\norm{x_k}\in\partial^\star f(x_k)\cap\Sphere$ for each $k\in\nat$ until $x_k$ reaches the solution.
Then, this setting satisfies Assumption~\ref{assum:0}.
\end{example}
Here we should remark that Assumption~\ref{assum:0} does not guarantee that the sequence generated by Algorithm~\ref{alg:main} converges to some optimum.
For example, let us consider the case where $f(x):=\min\{\abs{x},1\}$ for any real $x\in\real$ and the initial point $x_1:=3/2$.
Even though it violates the assumption of Example~\ref{example:2}, let us assume $v_k:=2$ for all $k\in\nat$.
Then, this setting still satisfies Assumption~\ref{assum:0}.
However, as illustrated in Figure~\ref{fig:zigzag}, we can see that the generated sequence does not converge to an optimum.
\begin{figure}[htbp]
  \centering
  \includegraphics[width=4truecm]{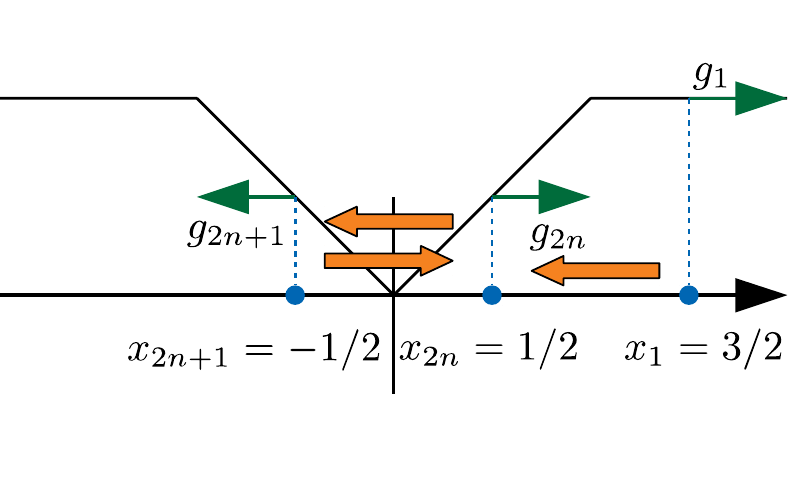}
  \caption{Illustration of case where the generated sequence repeats $1/2$ and $-1/2$ and does not converge to the origin.}\label{fig:zigzag}
\end{figure}
In each step, the approximation is moved in the direction of the origin by $1$.
In this counterexample, the first step moves the initial point $3/2$ to the point $1/2$.
After that, the algorithm eternally iterates so as to move approximations to their symmetric point with respect to the origin.
Therefore, the generated sequence repeats $1/2$ and $-1/2$ and does not converge to the optimum $0$. The convergence theorems presented later describes what is required to make the generated sequence converge to the optimum and/or how much error can occur in the solution.

Finally, we present a concrete application that can be dealt with as a quasiconvex optimization problem.
The following fractional programming problem is called the Cobb-Douglas production efficiency problem and satisfies Assumption~\ref{assum:0}.
\begin{example}[{\cite[Problem~(3.13)]{bradley1974}, \cite[Problem~(6.1))]{yaohua2015}, \cite[Section~1.7]{stancu1997}}]\label{example:3}
Let us consider the problem in Euclidean space; i.e., suppose that $\Hsp:=\real^n$.
Set $D:=[0,M]^n$ for some $M>0$, and set $\alpha_k:=1/2$ for any $k\in\nat$.
We give two positive scalars $a_0,c_0>0$ and two $n$-dimensional positive vectors in advance, $a,c\in(0,\infty)^n$ such that $\sum_{i=1}^na_i=1$.
Let
\begin{align*}
f(x):=\begin{cases}
\frac{-a_0\prod_{j=1}^nx_j^{a_j}}{\ip{c,x}+c_0}&(x\in[0,\infty)^n),\\
0&(\text{\rm otherwise}),
\end{cases}
\end{align*}
and assume that $T$ is a nonexpansive mapping from $\real^n$ to itself and $\Fix(T)\cap D\neq\emptyset$ holds.
Run Algorithm~\ref{alg:main} with an initial point $x_1\in[0,M]^n$.
Then, Assumption~\ref{assum:0} holds.
\end{example}
In Section~\ref{sec:experiment}, we will define the Cobb-Douglas production efficiency problem and show the numerical behavior of Algorithm~\ref{alg:main} when it solves a concrete instance of this problem.
Hence, we will put off explaining the problem in detail till later and limit ourselves here to a brief description.

The Cobb-Douglas production efficiency problem was introduced by Bradley and Frey \cite{bradley1974}.
Hu, Yang, and Sim proposed an algorithm for solving this problem by regarding it as a constrained quasiconvex optimization problem \cite{yaohua2015}.
The fact that the objective functional $f$ is quasiconvex allows us to treat it as a quasiconvex optimization problem.
The study by Hu et al.~\cite{yaohua2015} is also based on this fact; it ensures the quasiconvexity of the objective functional.
However, the existing results including that of \cite{yaohua2015} assume that the projection onto the constraint set is easily computable.

Let us consider the meaning of this objective function.
The numerator expresses the total profit defined by the production factors $x_j$ for each $j=1,2,\ldots,n$.
This numerator is modeled with the Cobb-Douglas production function.
In this problem, we consider the total cost for the production activities to be an affine function with respect to the production factors $\{x_j\}_{j=1}^n$.
This cost function is set as the denominator of the objective function.
Hence, the objective function $f$ represents the ratio of the total profit and the total cost.
In addition, it is known that the numerator, i.e., the Cobb-Douglas production function, is convex \cite[Section~2]{lee2013}, and therefore, Proposition~\ref{prop:fraction} guarantees that $f$ is quasiconvex.

The sequence generated by Algorithm~\ref{alg:main} must be contained in the box $[0,M]^n$, because its complement includes points that make the denominator of the function $f$ zero.
Therefore, we set the domain to $D=[0,M]^n\subset[0,\infty)^n$.
However, from the definition of $f$, setting $f(x):=0$ when $x$ is out of the set $[0,\infty)^n$ makes it possible to expand its domain to the whole space while maintaining continuity.
\begin{comment}
This is an example of an appropriate approximation of the objective function described on page~\pageref{pg:assum}.
\end{comment}

The following lemmas show the fundamental properties of Algorithm~\ref{alg:main}.
\begin{lem}\label{lem:fun}
Let $\{x_k\}\subset\Hsp$ be a sequence generated by Algorithm~\ref{alg:main}.
Suppose that Assumptions~\ref{assum:basic} and (A\ref{assum:holder}) hold.
Then, for any $k\in\nat$ that satisfies $f_\star<f(x_k)$, the following inequality holds.
\begin{align*}
\norm{x_{k+1}-x^\star}^2
\le\norm{x_k-x^\star}^2-2v_k(1-\alpha_k)\left(\frac{f(x_k)-f_\star}{L}\right)^\frac{1}{\beta}+(1-\alpha_k)v_k^2.
\end{align*}
\end{lem}
\label{M:4}
  Similar lemmas are presented in the literature that discuss algorithms which use the metric projection onto the constraint set \cite[Lemma~3.1]{yaohua2016b}, \cite[Lemma~3.2]{yaohua2016}, \cite[Lemma~6]{kiwiel2001}.
  This implies that the generated sequence is guaranteed to be contained in the constraint set.
  Therefore, the proofs of these lemmas use a property which ensures $x_k\in X\cap D$ and $f_\star\le f(x_k)$.
  However, the algorithm presented here generates a sequence that may not be in the fixed point set of $T$, in other words, it may be out of the feasible set.
  The convergence analyses are carefully divided into cases where $f_\star<f(x_k)$ holds and cases where it does not hold.
  The above lemma describes that the existing results  hold for the proposed algorithm only if the positive case $f_\star<f(x_k)$ holds and even if the containedness of the generated sequence in the fixed point set cannot be guaranteed.

\begin{proof}\label{proof:lem:fun}
  See {\ifextended\ref{appendix:lemmasproofs}\else Appendix~C\fi} for the proof of this lemma.
\end{proof}

\begin{lem}\label{lem:dist}
Let $\{x_k\}\subset\Hsp$ be a sequence generated by Algorithm~\ref{alg:main}.
Suppose that Assumptions~\ref{assum:basic} and (A\ref{assum:bounded}) hold and the real sequence $\{v_k\}$ is bounded.
Then, for each $x\in\Fix(T)\cap D$, there exists $M_1\ge 0$ such that
\begin{align*}
\norm{x_{k+1}-x}^2
\le\norm{x_k-x}^2-(1-\alpha_k)\norm{x_k-T(x_k-v_kg_k)}^2+v_kM_1\quad(k\in\nat).
\end{align*}
\end{lem}

\begin{proof}\label{proof:lem:dist}
  See {\ifextended\ref{appendix:lemmasproofs}\else Appendix~C\fi} for the proof of this lemma.
\end{proof}

\subsection{Constant step-size rule}
The following theorem shows how precise the generated solution is when the constant step-size rule is used.
\begin{thm}\label{thm:constant}
Let $v>0$ and $v_k:=v$ for all $k\in\nat$ and $\{x_k\}\subset\Hsp$ be a sequence generated by Algorithm~\ref{alg:main}.
Suppose that Assumptions~\ref{assum:basic} and \ref{assum:0} hold.
Then, the sequence $\{x_k\}$ satisfies
\begin{align*}
\liminf_{k\to\infty}f(x_k)\le f_\star+L\left(\frac{v}{2}\right)^\beta
\text{, and }\liminf_{k\to\infty}\norm{x_k-T(x_k)}^2\le Mv
\end{align*}
for some $M\ge 0$.
\end{thm}
\begin{proof}
See {\ifextended\ref{appendix:constantproof}\else Appendix~D\fi} for the proof of this theorem.
\end{proof}

\subsection{Diminishing step-size rule}
Finally, we prove the weak convergence theorem of Algorithm~\ref{alg:main}.
To let the generated sequence weakly converge to some optimum, we can use a specific step-size called a diminishing step-size.
The following theorem describes this condition and shows that the generated sequence weakly converges.

\begin{thm}\label{thm:dsr}
Let $\{x_k\}\subset\Hsp$ be a sequence generated by Algorithm~\ref{alg:main}.
Suppose that
\begin{enumerate}[(i)]
\item
Assumptions~\ref{assum:basic} and \ref{assum:0} hold,
\item
and the real sequence $\{v_k\}\subset(0,\infty)$ satisfies
\begin{align*}
\lim_{k\to\infty}v_k=0
\text{, and }
\sum_{k=1}^\infty v_k=\infty.
\end{align*}
\end{enumerate}
Then, there exists a subsequence of the generated sequence $\{x_k\}$ which converges weakly to a point in $X^\star$.
In addition, if
\begin{enumerate}[(i)]
\setcounter{enumi}{2}
\item
the whole space $\Hsp$ is an $N$-dimensional Euclidean space $\real^N$,
\item\label{M:14}
and the solution $x^\star\in X^\star$ is unique,
\end{enumerate}
then, the whole sequence $\{x_k\}$ converges to this unique solution $x^\star$.
\end{thm}
Assumption~(A\ref{a3:nonempty}) and Proposition~\ref{prop:uniqueness} show that the strict quasiconvexity of the objective function is a sufficient condition for the uniqueness of the solution $x^\star\in X^\star$.

Before proving the above theorem, we prove the following lemma which will be needed later.
\begin{lem}\label{lem:dsr}
Suppose that Assumptions~\ref{assum:basic} and \ref{assum:0} hold, and suppose that the real sequence $\{v_k\}\subset(0,\infty)$ satisfies
\begin{align*}
\lim_{k\to\infty}v_k=0
\text{, and }
\sum_{k=1}^\infty v_k=\infty.
\end{align*}
Let $\{x_k\}\subset\Hsp$ be the sequence generated by Algorithm~\ref{alg:main} with this real sequence $\{v_k\}$.
Then,
\begin{align*}
\liminf_{k\to\infty}f(x_k)\le f_\star
\end{align*}
holds.
\end{lem}
\begin{proof}
See {\ifextended\ref{appendix:dsrlemma}\else Appendix~E\fi} for the proof of this theorem.
\end{proof}

Now let us prove Theorem~\ref{thm:dsr} with the above result.
\begin{proof}[Proof of Theorem~\ref{thm:dsr}]
Let the limit superior of the real sequence $\{\alpha_k\}$ be denoted by $\bar{\alpha}\in(0,1)$.
Fix $x^\star\in X^\star$ arbitrarily.
We will prove the assertion by separating the problem into two cases: the case where there exists a number $k_0\in\nat$ such that $\norm{x_{k+1}-x^\star}\le\norm{x_k-x^\star}$ for all $k\ge k_0$, and its negation.

First, let us consider the positive case; i.e., there exists a number $k_0\in\nat$ such that $\norm{x_{k+1}-x^\star}\le\norm{x_k-x^\star}$ for all $k\ge k_0$.
The property of the limit superior guarantees the existence of a number $k_1\ge k_0$ such that $\alpha_k<\bar{\alpha}+(1-\bar{\alpha})/2$ for all $k\ge k_1$.
Therefore, applying this relationship between $\bar{\alpha}$ and $\alpha_k$ to Lemma~\ref{lem:dist}, an estimate of $\norm{x_k-T(x_k-v_kg_k)}$ for any $k\ge k_1$ can be obtained as follows:
\begin{align}
\frac{1}{2}(1-\bar{\alpha})\norm{x_k-T(x_k-v_kg_k)}^2
&\le(1-\alpha_k)\norm{x_k-T(x_k-v_kg_k)}^2\nonumber\\
&\le\norm{x_k-x^\star}^2-\norm{x_{k+1}-x^\star}^2+v_kM_1.\label{ineq:d1}
\end{align}
In this case, the monotonicity and boundedness of the subsequence $\{\norm{x_k-x^\star}\}_{k\ge k_1}$ are assured.
Hence, this subsequence converges to some nonnegative real.
Since we have assumed that the real sequence $\{v_k\}$ converges to zero, the left-hand side of inequality~\eqref{ineq:d1} converges to zero.
On the other hand, $\norm{x_k-T(x_k)}$ for each $k\in\nat$ can be expanded with the triangle inequality and by noting the nonexpansivity of $T$ as follows:
\begin{align*}
\norm{x_k-T(x_k)}
&\le\norm{x_k-T(x_k-v_kg_k)}+\norm{T(x_k-v_kg_k)-T(x_k)}\\
&\le\norm{x_k-T(x_k-v_kg_k)}+v_k.
\end{align*}
From the assumption of this theorem and the previous discussion, both terms on the right-hand side above converge to zero with respect to $k$.
Hence, we find that the real sequence $\{\norm{x_k-T(x_k)}\}$ converges to zero.

The property of the limit inferior of the real sequence $\{f(x_k)\}$ guarantees the existence of a subsequence $\{f(x_{k_i})\}$ converging to $\liminf_{k\to\infty}f(x_k)$.
Note that Lemma~\ref{lem:dsr} asserts that this limit inferior is less than or equal to the minimum value $f_\star$.
There exist a point $u\in\Hsp$ and a subsequence $\{x_{k_{i_j}}\}\subset\{x_{k_i}\}$ converging weakly to the point $u$, since $\{x_{k_i}\}$ is a bounded sequence in $\Hsp$.
Now, suppose that $u$ is not a fixed point of $T$.
Since the real sequence $\{\norm{x_k-T(x_k)}\}$ converges to zero and $T$ is a nonexpansive mapping, Proposition~\ref{prop:opial} produces a contradiction as follows:
\begin{align*}
\liminf_{j\to\infty}\norm{x_{k_{i_j}}-u}
&<\liminf_{j\to\infty}\norm{x_{k_{i_j}}-T(u)}\\
&\le\liminf_{j\to\infty}\left(\norm{x_{k_{i_j}}-T\left(x_{k_{i_j}}\right)}+\norm{T(x_{k_{i_j}})-T(u)}\right)\\
&\le\liminf_{j\to\infty}\norm{x_{k_{i_j}}-u}.
\end{align*}
Therefore, we can see that $u$ is a fixed point of $T$.
Proposition~\ref{prop:weaklycontinuous} means that the objective functional $f$ has weakly lower semicontinuity.
The weak convergence of the sequence $\{x_{k_{i_j}}\}$ implies
\begin{align*}
f(u)
\le\liminf_{j\to\infty}f(x_{k_{i_j}})
=\lim_{i\to\infty}f(x_{k_i})
\le f_\star;
\end{align*}
that is, $u$ is an optimum.

To deal with the positive case, let us consider the weak convergence of $\{x_k\}$ and its subsequences.
Take another subsequence $\{x_{k_{i_l}}\}\subset\{x_{k_i}\}$ that converges weakly to a point $v\in\Hsp$.
A similar discussion to the one for obtaining $u\in X^\star$ ensures that the point $v$ is also an optimum.
To show the uniqueness of the weak accumulation points of the sequence $\{x_{k_i}\}$, let us assume that $u\neq v$.
Since the sequence $\{\norm{x_k-x^\star}\}$ converges, Proposition~\ref{prop:opial} leads us to a contradiction:
\begin{align*}
\lim_{k\to\infty}\norm{x_k-u}
&=\lim_{j\to\infty}\norm{x_{k_{i_j}}-u}
<\lim_{j\to\infty}\norm{x_{k_{i_j}}-v}\\
&=\lim_{j\to\infty}\norm{x_k-v}
=\lim_{l\to\infty}\norm{x_{k_{i_l}}-v}
<\lim_{l\to\infty}\norm{x_{k_{i_l}}-u}\\
&=\lim_{k\to\infty}\norm{x_k-u}.
\end{align*}
Hence, we can see that $u$ is the same point as $v$, and the uniqueness of all weak accumulation points of the sequence $\{x_{k_i}\}$ is proven.
This uniqueness implies that the sequence $\{x_{k_i}\}$ converges weakly to $u\in X^\star$.
Now take another subsequence $\{x_{k_m}\}\subset\{x_k\}$ that converges weakly to a point $w\in\Hsp$, and suppose that $u$ is a different point from $v$.
From the fact that the sequence $\{\norm{x_k-x^\star}\}$ converges and from Proposition~\ref{prop:opial}, we can deduce that
\begin{align*}
\lim_{k\to\infty}\norm{x_k-u}
&=\lim_{i\to\infty}\norm{x_{k_i}-u}
<\lim_{i\to\infty}\norm{x_{k_i}-w}\\
&=\lim_{j\to\infty}\norm{x_k-w}
=\lim_{m\to\infty}\norm{x_{k_m}-w}
<\lim_{m\to\infty}\norm{x_{k_m}-u}\\
&=\lim_{k\to\infty}\norm{x_k-u}.
\end{align*}
However, this is a contradiction.
Hence, the sequence $\{x_k\}$ converges weakly to some optimum.
This proves the positive case.

Next, let us consider the negative case, in other words, the case where a subsequence $\{x_{k_i}\}\subset\{x_k\}$ exists that satisfies $\norm{x_{k_i}-x^\star}<\norm{x_{k_i+1}-x^\star}$ for all $i\in\nat$.
A similar discussion to the one for finding the number $k_1$ in the positive case guarantees the existence of $i_0\in\nat$ satisfying $\alpha_{k_i}<\bar{\alpha}+(1-\bar{\alpha})/2$ for all $i\ge i_0$.
The distances from the point $x^\star$ to each point $x_{{k_i}+1}$ where $i\ge i_0$ from Lemma~\ref{lem:dist} are as follows:
\begin{align*}
\norm{x_{k_i+1}-x}^2
&\le\norm{x_{k_i}-x}^2-(1-\alpha_{k_i})\norm{x_{k_i}-T(x_{k_i}-v_{k_i}g_{k_i})}^2+v_{k_i}M_1\\
&\le\norm{x_{k_i}-x}^2-\frac{1}{2}(1-\bar{\alpha})\norm{x_{k_i}-T(x_{k_i}-v_{k_i}g_{k_i})}^2+v_{k_i}M_1.
\end{align*}
Here, we have assumed that $\norm{x_{k_i+1}-x^\star}$ is greater than $\norm{x_{k_i}-x^\star}$ for all $i\in\nat$ and the real sequence $\{v_k\}$ converges to zero.
Thus, the above inequality implies that the real sequence $\{\norm{x_{k_i}-T(x_{k_i}-v_{k_i}g_{k_i})}\}$ converges to zero with respect to $i$.
On the other hand, the distances $\norm{x_{k_i}-T(x_{k_i})}$ for each $i\in\nat$ can be estimated from the nonexpansivity of $T$ as follows:
\begin{align*}
\norm{x_{k_i}-T(x_{k_i})}
&\le\norm{x_{k_i}-T(x_{k_i}-v_{k_i}g_{k_i})}+\norm{T(x_{k_i}-v_{k_i}g_{k_i})-T(x_{k_i})}\\
&\le\norm{x_{k_i}-T(x_{k_i}-v_{k_i}g_{k_i})}+v_{k_i}.
\end{align*}
Since both terms of the right-hand side of the above inequality converge to zero, the left-hand side also converges to zero.

We will proceed by way of contradiction; suppose that $\limsup_{i\to\infty}f(x_{k_i})>f_\star$.
This implies the existence of $\delta>0$ and a subsequence $\{x_{k_{i_j}}\}\subset\{x_{k_i}\}$ such that $f_\star+\delta<f(x_{k_{i_j}})$ for all $j\in\nat$.
Since $\norm{x_{k_i}-x^\star}<\norm{x_{k_i+1}-x^\star}$ and $f_\star<f(x_{k_{i_j}})$ hold for any $j\in\nat$, we can use Lemma~\ref{lem:fun} to get
\begin{align*}
v_{k_{i_j}}\left(1-\alpha_{k_{i_j}}\right)\left(2\left(\frac{\delta}{L}\right)^{\frac{1}{\beta}}-v_{k_{i_j}}\right)<0
\end{align*}
for all $j\in\nat$.
The above inequality does not hold for sufficiently large $j\in\nat$, since the real sequence $\{v_k\}$ converges to zero.
Therefore, we arrive at a contradiction, and thus, $\limsup_{i\to\infty}f(x_{k_i})\le f_\star$.

The boundedness of the sequence $\{x_{k_i}\}$ guarantees the existence of a subsequence $\{x_{k_{i_j}}\}\subset\{x_{k_i}\}$ that weakly converges to some point $u\in\Hsp$.
To show that $u$ is a fixed point of the mapping $T$, let us assume that it is not.
Recall that the real sequence $\{\norm{x_{k_i}-T(x_{k_i})}\}$ converges to zero.
Hence, the nonexpansivity of $T$ together with Proposition~\ref{prop:opial} produces a contradiction,
\begin{align*}
\liminf_{j\to\infty}\norm{x_{k_{i_j}}-u}
&<\liminf_{j\to\infty}\norm{x_{k_{i_j}}-T(u)}\\
&\le\liminf_{j\to\infty}\left(\norm{x_{k_{i_j}}-T(x_{k_{i_{j_l}}})}+\norm{T(x_{k_{i_j}})-T(u)}\right)\\
&\le\liminf_{j\to\infty}\norm{x_{k_{i_j}}-u}.
\end{align*}
Therefore, we have $u\in\Fix(T)$. In addition, Proposition~\ref{prop:weaklycontinuous} means that the objective functional $f$ has weakly lower semicontinuity.
Hence,
\begin{align*}
f(u)
\le\liminf_{j\to\infty}f(x_{k_{i_j}})
\le\limsup_{i\to\infty}f(x_{k_i})
\le f_\star
\end{align*}
holds.
This implies conclusively that there exists a subsequence of $\{x_k\}$ which weakly converges to the optimum $u\in X^\star$.

We omit the proof of the additional part of this theorem.
The complete proof is in {\ifextended\ref{proof:thm:dsr}\else Appendix~F\fi}.
\end{proof}

%%%%%%%%%%%%%%%%%%%%%%%%%%%%%%%%%%%%%%%%%%%%%%%%%%%%%%%%%%%%%%%%%%%%%%%%%%%%%%%
%%% EXPERIMENTS
%%%%%%%%%%%%%%%%%%%%%%%%%%%%%%%%%%%%%%%%%%%%%%%%%%%%%%%%%%%%%%%%%%%%%%%%%%%%%%%
\section{Numerical Experiments}\label{sec:experiment}
To confirm that Algorithm~\ref{alg:main} converges to the optimum and evaluate its performance, we ran it and an existing algorithm \cite[Algorithm~(14)]{kiwiel2001} on a concrete constrained quasiconvex optimization problem, i.e., the Cobb-Douglas production efficiency problem \cite[Problem~(3.13)]{bradley1974}, \cite[Problem~(6.1)]{yaohua2015}, \cite[Section~1.7]{stancu1997}.
Let us redefine Example~\ref{example:3} with concrete constraints as follows.
\begin{problem}[{\cite[Problem~(3.13)]{bradley1974}, \cite[Problem~(6.1)]{yaohua2015}, \cite[Section~1.7]{stancu1997}}]\label{prob:experiment}
Suppose that $\Hsp:=\real^n$.
Let $a_0,c_0>0$ and let $a,c\in(0,\infty)^n$ such that $\sum_{i=1}^na_i=1$.
Furthermore, let $b_i\in[0,\infty)^n$, $\underline{p}_i\in[0,\infty)^n$, and $\overline{p}_i\in(0,\infty]^n$ for $i=1,2,\ldots,m$.
Then, we would like to
\begin{align*}
\text{minimize }
& f(x):=\begin{cases}
\frac{-a_0\prod_{j=1}^nx_j^{a_j}}{\ip{c,x}+c_0}&(x\in[0,\infty)^n),\\
0&(\text{\rm otherwise}),
\end{cases}\\
\text{subject to }
& \underline{p}_i\le\ip{b_i,x}\le\overline{p}_i\quad(i=1,2,\ldots,m),\\
& x\in D:=[0,M]^n,
\end{align*}
where $M>0$.
\end{problem}
This problem is an instance of Example~\ref{example:3}.
Indeed, metric projections onto any closed half spaces, including boxes, can be computed explicitly \cite[Example~29.20]{bc}, and the transformations~(T\ref{t:average}--\ref{t:firmup}) enable us to build a firmly nonexpansive mapping whose fixed point set coincides with their intersection.
Our GitHub repository, \GitHubURL{}, provides the means to make a metric projection onto a given half space and the transformations (T\ref{t:average}--\ref{t:firmup}).
We used these implementation in the following experiments.

Before discussing our experiments, let us examine the background of this problem.
The goal is to find the most efficient production factors under funding-level restrictions \cite[Section~6]{yaohua2015}.
As mentioned in Section~\ref{sec:convergence}, the objective function $f$ represents the ratio between the total profit (what is obtained) and the total cost (how much expenditure is required) as an efficiency indicator.
We also described how the total profit and the total cost are modeled in Section~\ref{sec:convergence}.
The total profit is the numerator of the objective function and is modeled with the Cobb-Douglas production function on the production factors $x\in\real^n$.
The total cost is the denominator of the objective function and is modeled with the affine function on the production factors $x\in\real^n$.
There are a variety of constraints on the funding level \cite[Section~6]{yaohua2015}.
These constraints represent the duties and restrictions of each production project $i=1,2,\ldots,m$.
These indicators are modeled with affine functions and we set two parameters $\underline{p}_i, \overline{q}_i$ as lower and upper bounding constraints to the indicator of each project $i=1,2,\ldots,m$.

We conducted numerical experiments in three cases.
First, in the unbounded constraint case, which is treated in the existing literature \cite{yaohua2015}, we set $M:=100$ and did not guarantee the uniqueness of the optima, which is required for letting the generated sequence converge.
Second, in the bounded constraint case, we set $M:=100$ and guaranteed the uniqueness of optima, as shown in Section~\ref{sec:convergence}.
In practice, we cannot manufacture products infinitely because there are many restrictions on the amount of materials, capital, human resources, number and/or capacity of machines, environments, and so on.
Therefore, this case has realistic experimental assumptions for optimizing production efficiency.
Furthermore, we conducted an optimization over the generalized convex feasible sets.

We compared Algorithm 1 with the exact quasi-subgradient method (QSM) \cite[Algorithm~(14)]{kiwiel2001}.
\begin{comment}
\begin{algorithm}
  \caption{The exact quasi-subgradient method (QSM) \cite[Algorithm~(14)]{kiwiel2001}}\label{alg:qsm}
  \begin{algorithmic}[1]
    \Require
      \Statex{$f$: $\real^n\to\real$, $X\cap D\subset\real^n$: the feasible set;}
      \Statex{$P_X$: $\real^n\to\real^n$, the metric projection onto the constraint set $X$;}
      \Statex{$\{v_k\}\subset(0,\infty)$, $\{\alpha_k\}\subset(0,1]$.}
    \Ensure
      \Statex{$\{x_k\}\subset{}\real^n$.}
    \State{$x_1\in\real^n$.}
    \For{$k=1,2,\ldots$}
      \State{$g_k\in\partial^\star f(x_k)\cap\Sphere$.}\label{alg:main:1}
      \State{$x_{k+1}:=P_{X\cap D}(x_k-v_kg_k)$.}\label{alg:main:improve}
    \EndFor
  \end{algorithmic}
\end{algorithm}
\end{comment}
In order for it to run, this algorithm requires a computation of the metric projection onto the feasible set $X\cap D$. Here, we used a trust-region algorithm for constrained optimization, \texttt{trust-constr}, implemented as the \texttt{scipy.optimize.minimize} solver provided by the SciPy fundamental library for scientific computing \cite{scipy}.
This algorithm also requires the step-sizes $\{v_k\}$ for it to run.
We set its error tolerance a tenth of $v_k$ for each $k=1,2,\ldots$.
That is, we solved the subproblem to
\begin{align*}
\text{find } & \norm{x_{k+1}-(x_k-v_kg_k)}^2\le\min_{u\in X}\norm{u-(x_k-v_kg_k)}^2+\frac{v_k}{10}\\
\text{subject to } & x_{k+1}\in X,
\end{align*}
as the computation of the metric projection in step~\ref{alg:main:improve} with the existing optimization solver.

\label{M:2}
  In contrast to the existing algorithm which finds the metric projection onto the constraint set in the above way, we constructed a firmly nonexpansive mapping whose fixed point set coincides with the constraint set and gave it to Algorithm~\ref{alg:main}.
  Here, the constraint set is the intersection of the half-spaces $\{x:\underline{p}_i\le\ip{b_i,x}\}$ and $\{x:\ip{b_i,x}\le\overline{p}_i\}$ for $i=1,2,\ldots,m$.
  The metric projection onto each half-space can be easily computed \cite[Example~29.20]{bc}:
  \begin{align*}
    P_{\{x:\underline{p}_i\le\ip{b_i,x}\}}(x):=\begin{cases}
      x & (\underline{p}_i\le\ip{b_i,x}), \\
      x+\frac{\underline{p}_i-\ip{b_i,x}}{\norm{b_i}^2}b_i & (\text{otherwise})
    \end{cases},\\
    P_{\{x:\ip{b_i,x}\le\overline{p}_i\}}(x):=\begin{cases}
      x & (\ip{b_i,x}\le\overline{p_i}), \\
      x+\frac{\overline{p}_i-\ip{b_i,x}}{\norm{b_i}^2}b_i & (\text{otherwise})
    \end{cases}
  \end{align*}
  for all $i=1,2,\ldots,m$ and for any $x\in\Hsp$.
  To construct a nonexpansive mapping whose fixed point set coincides with the intersection of the above sets, we use the transformation (T\ref{t:average}) and construct
  \begin{align*}
    \tilde{T}(x):=\frac{1}{m}\sum_{i=1}^m\frac{P_{\{x:\underline{p}_i\le\ip{b_i,x}\}}(x)+P_{\{x:\ip{b_i,x}\le\overline{p}_i\}}(x)}{2}
  \end{align*}
  for any $x\in\Hsp$.
  This mapping $\tilde{T}$ is nonexpansive, but not firmly nonexpansive.
  Therefore, we convert the nonexpansive mapping $\tilde{T}$ into the corresponding firmly nonexpansive one $T$ by using the transformation (T\ref{t:firmup}):
  \begin{align*}
    T:=\frac{\Id+\tilde{T}}{2}.
  \end{align*}
  We gave $T$ to Algorithm~\ref{alg:main} in the experiment.

Our experimental environment was as follows: Python~3.6.6 with NumPy~1.15.0 \cite{numpy} and SciPy~1.1.0 \cite{scipy} libraries on macOS High Sierra version 10.13.6 on Mac Pro (Late 2013) with a 3 GHz 8 Cores Intel Xeon E5 CPU and 32GB 1800MHz DDR3 memory.
We used the \texttt{time.process\_time} method for the evaluating computational time of each algorithm.
The method was implemented with the \texttt{clock\_gettime(2)} system call and had a $10^{-6}$ second resolution.
Our GitHub repository, \GitHubURL{}, provides the codes that were used in the experiments.
It has the implementations of Algorithm~\ref{alg:main} and QSM and miscellaneous utilities including higher-order functions to be used for composing a nonexpansive mapping.
% In our experiments, we used not only the diminishing step-size rule which conforms to the assumption of Theorem~\ref{thm:dsr} but also the constant step-size rule which sets step-sizes some constant value and is frequently used in the existing researches\cite{yaohua2015,iiduka2015,iiduka2016,iiduka2016oms} as a useful approximating method.

We ran Algorithm~\ref{alg:main} and QSM with five different randomly chosen initial points, limited their computational time to ten seconds, and evaluated the average of the computed number of iterations $k$ and the following values:
\begin{align*}
V_\text{func}:=\frac{1}{8}\sum_{i=1}^8f(x^\star_{(i)}),\quad
V_\text{dist}:=\frac{1}{8}\sum_{i=1}^8\norm{x^\star_{(i)}-T(x^\star_{(i)})},
\end{align*}
where $x^\star_{(i)}$ is the solution obtained for each sampling $i=1,2,\ldots,8$.

\subsection{Unbounded constraint case}
Here, we ran Algorithm~\ref{alg:main} and QSM on Problem~\ref{prob:experiment} with the following settings:
$n:=100$; $m:=100$;
$a_0,c_0\in(0,10]$, $\tilde{a}\in(0,1]^n$, $c\in(0,10]^n$ were chosen randomly; $a:=\tilde{a}/\sum_{i=1}^na_i$;
$b_i\in[0,1)^n$, $\underline{p}_i\in[0,25\norm{b_i})$ were chosen randomly for each $i=1,2,\ldots,m$; $\overline{p}_i:=\infty$ for all $i=1,2,\ldots,m$;
and $M:=+\infty$.

The experimental results are shown in Table~\ref{table:ubdd}.
\begin{table}[htbp]
  \centering
  \caption{Results of unbounded constraint case.}\label{table:ubdd}
  \begin{tabular}{l|rrr}
    \toprule
    & $k$ & $V_\text{func}$ & $V_\text{dist}$ \\
    \midrule
    Alg.~\ref{alg:main} ($v_k:=10^{-1}$) & 11495.5 & -0.00092189 & $3.18571477\times 10^{-13}$ \\
    Alg.~\ref{alg:main} ($v_k:=10^{-2}$) & 11539.5 & -0.00061939 & $3.17169706\times 10^{-13}$ \\
    Alg.~\ref{alg:main} ($v_k:=10^{-3}$) & 11459.2 & -0.00049238 & $3.10991757\times 10^{-13}$ \\
    Alg.~\ref{alg:main} ($v_k:=10^{-1}/k$) & 11474.5 & -0.00046070 & $3.05626418\times 10^{-13}$ \\
    Alg.~\ref{alg:main} ($v_k:=10^{-2}/k$) & 11264.5 & -0.00045498 & $3.16790812\times 10^{-13}$ \\
    Alg.~\ref{alg:main} ($v_k:=10^{-3}/k$) & 11318.4 & -0.00045426 & $3.28061131\times 10^{-13}$ \\
    QSM ($v_k:=10^{-1}$) & 160.4 & -0.00050596 & $4.04768150\times 10^{-13}$ \\
    QSM ($v_k:=10^{-2}$) & 39.6 & -0.00045707 & $4.12640954\times 10^{-13}$ \\
    QSM ($v_k:=10^{-3}$) & 25.1 & -0.00045306 & $4.00174170\times 10^{-13}$ \\
    QSM ($v_k:=10^{-1}/k$) & 36.1 & -0.00045750 & $4.23263890\times 10^{-13}$ \\
    QSM ($v_k:=10^{-2}/k$) & 26.6 & -0.00045353 & $3.99258739\times 10^{-13}$ \\
    QSM ($v_k:=10^{-3}/k$) & 20.5 & -0.00045275 & $4.10328799\times 10^{-13}$ \\
    \bottomrule
  \end{tabular}
\end{table}
The proposed algorithm (Algorithm~\ref{alg:main}) can iterate the computation more times than the existing one within the same computational time.
Algorithm~\ref{alg:main} does not require any subproblem to be solved, while QSM requires one to be solved in order to find a metric projection onto the constraint set.
Therefore, the required time for computing an iteration of Algorithm~\ref{alg:main} is much less than that of QSM.
According to the values of $D$, both Algorithm~\ref{alg:main} and QSM for any step-size (and no matter whether a constant or diminishing step-size rule was used) can obtain the solution belonging to the constraint set.
Indeed, our experimental environment (NumPy) used the \texttt{float64} data type (double precision float: sign bit, 11-bit exponent, and 52-bit mantissa) to express a real number, and its resolution is $10^{-15}$.
By considering the number of dimensions as well, we can regard all values of $D$ to be almost zero.
Let us examine the functional values of the obtained solutions.
When we applied Algorithm~\ref{alg:main} and QSM to the problem with the same step-size, we found that the function value of the solution obtained by Algorithm~\ref{alg:main} is better than that of QSM.
In particular, the function value obtained by Algorithm~\ref{alg:main} with $v_k:=10^{-1}$ is nearly twice as good as QSM with the same step-size.
Since Algorithm~\ref{alg:main} can iterate the main loop more times than QSM, it can reduce the functional value sufficiently.

\subsection{Bounded constraint case}
Next, we evaluated Algorithm~\ref{alg:main} and QSM when they were run with the following settings:
$n:=100$; $m:=100$;
$a_0,c_0\in(0,10]$, $\tilde{a}\in(0,1]^n$, $c\in(0,10]^n$ were chosen randomly; $a:=\tilde{a}/\sum_{i=1}^na_i$;
$b_i\in[0,1)^n$, $\underline{p}_i\in[0,25\norm{b_i})$, $\overline{p}_i\in(75\norm{b_i},100\norm{b_i}]$ were chosen randomly for each $i=1,2,\ldots,m$;
and $M:=100$.
As shown in Example~\ref{example:3}, this case satisfies Assumption~\ref{assum:0}.
Therefore, the sequence generated by Algorithm~\ref{alg:main} is guaranteed to converge to some optimum.

The experimental results are shown in Table~\ref{table:bdd}.
\begin{table}[htbp]
  \centering
  \caption{Results of bounded constraint case.}\label{table:bdd}
  \begin{tabular}{l|rrr}
    \toprule
    & $k$ & $V_\text{func}$ & $V_\text{dist}$ \\
    \midrule
    Alg.~\ref{alg:main} ($v_k:=10^{-1}$) & 6254.0 & -0.00092536 & $1.21925957\times 10^{-13}$ \\
    Alg.~\ref{alg:main} ($v_k:=10^{-2}$) & 6208.9 & -0.00050503 & $5.22534774\times 10^{-3}$ \\
    Alg.~\ref{alg:main} ($v_k:=10^{-3}$) & 6276.1 & -0.00019717 & $6.42589842\times 10^{-4}$ \\
    Alg.~\ref{alg:main} ($v_k:=10^{-1}/k$) & 6293.8 & -0.00014214 & $1.28281187\times 10^{-5}$ \\
    Alg.~\ref{alg:main} ($v_k:=10^{-2}/k$) & 6245.6 & -0.00014163 & $8.56673942\times 10^{-6}$ \\
    Alg.~\ref{alg:main} ($v_k:=10^{-3}/k$) & 6294.0 & -0.00014162 & $8.33884360\times 10^{-6}$ \\
    All Results of QSM & 0.0 & --- & --- \\
    \bottomrule
  \end{tabular}
\end{table}
The existing algorithm (QSM) with all step-size rules could not compute even one iteration within the time limit, 10 seconds.
QSM required about 15 seconds to compute the first iteration.
In this case, QSM must compute the metric projection onto the intersection of two hundred halfspaces and a box, but this intersection is too complex in shape to compute quickly.
Therefore, it could not deal with this instance.
In contrast, Algorithm~\ref{alg:main} solved this instance.
In particular, Algorithm~\ref{alg:main} with $v_k:=10^{-1}$ found the solution having the best function value and belonging to the constraint set.
Therefore, it can solve problems even if their constraint sets have complex shapes.

\subsection{Optimization over generalized convex feasible sets}\label{M:3}
Finally, let us consider the case in which conflicts of constraints exist, i.e., the intersection of the constraint sets may be empty.
Even in this case, the proposed algorithm can be used if the constraints are extended to generalized convex feasible sets, such as described in Section~\ref{sec:preliminaries}.

In the two previous subsections, we computed the metric projection with the constrained smooth optimization solver provided by the SciPy library.
However, it is difficult to find the metric projection onto the set of minimizers of the functional~\eqref{dfn:gfcs} due to the discontinuity of its Hessian matrix and the complexity of the problem.
Therefore, in this subsection, we will examine only the performance of Algorithm~\ref{alg:main}.

The settings of this experiment were as follows:
$n:=100$;
$m:=100$; $a_0,c_0\in(0,10]$, $\tilde{a}\in(0,1]^n$, $c\in(0,10]^n$ were chosen randomly;
$a:=\tilde{a}/\sum_{i=1}^na_i$; $b_i\in[0,1)^n$, $\underline{p}_i,\overline{p}_i\in[0,100\norm{b_i})$ were chosen randomly for each $i=1,2,\ldots,m$;
and $M:=+\infty$.
The existence of a constraint $i\in\{1,2,\ldots,m\}$ which satisfies $\overline{p}_i<\underline{p}_i$ was guaranteed.
This implies that the intersection of at least one pair of constraints is empty.
We used \cite[Definition~(9)]{iiduka2016ejor} for constructing a firmly nonexpansive mapping whose fixed point set coincides with the constraint set. As in the previous subsection, this case also satisfies Assumption~\ref{assum:0}.

The experimental results are shown in Table~\ref{table:bdd2}.
\begin{table}[htbp]
  \centering
  \caption{Results of bounded constraint case.}\label{table:bdd2}
  \begin{tabular}{l|rrr}
    \toprule
    & $k$ & $V_\text{func}$ & $V_\text{dist}$ \\
    \midrule
    Alg.~\ref{alg:main} ($v_k:=10^{-1}$) & 4988.9 & -0.00015650 & $2.48927983\times 10^{-1}$ \\
    Alg.~\ref{alg:main} ($v_k:=10^{-2}$) & 4914.4 & -0.00012585 & $2.58656058\times 10^{-1}$ \\
    Alg.~\ref{alg:main} ($v_k:=10^{-3}$) & 4833.6 & -0.00012377 & $2.59531605\times 10^{-1}$ \\
    Alg.~\ref{alg:main} ($v_k:=10^{-1}/k$) & 4823.4 & -0.00012363 & $2.59665629\times 10^{-1}$ \\
    Alg.~\ref{alg:main} ($v_k:=10^{-2}/k$) & 4818.4 & -0.00012378 & $2.59656373\times 10^{-1}$ \\
    Alg.~\ref{alg:main} ($v_k:=10^{-3}/k$) & 4773.6 & -0.00012380 & $2.59626252\times 10^{-1}$ \\
    \bottomrule
  \end{tabular}
\end{table}
Algorithm~\ref{alg:main} solved the problem similarly for each step-size rule; the results scarcely depended on the step-size rule.
With a constant step-size $v_k=10^{-1}$ for all $k\in\nat$, it gave the best score in terms of $V_\text{func}$ and $V_\text{dist}$.
Although the time required for computing one iteration exceeded those of the two previous experiments, it approximated the solution of this complicated problem in 10 seconds.

%%%%%%%%%%%%%%%%%%%%%%%%%%%%%%%%%%%%%%%%%%%%%%%%%%%%%%%%%%%%%%%%%%%%%%%%%%%%%%%
%%% CONCLUSION
%%%%%%%%%%%%%%%%%%%%%%%%%%%%%%%%%%%%%%%%%%%%%%%%%%%%%%%%%%%%%%%%%%%%%%%%%%%%%%%
\section{Conclusion}\label{sec:conclusion}% and Future Work}\label{sec:conclusion}
We proposed the novel algorithm for solving the constrained quasiconvex optimization problem even if the metric projection onto its constraint set cannot be computed easily.
We showed its convergence for constant and diminishing step-size rules.
When the step-size is constant, the limit inferiors of the functional value and the degree of approximation to the fixed point are guaranteed to be optimal and tolerate errors proportioned to the step-size.
When the step-size is diminishing, the existence of a subsequence of the generated sequence such that it converges to the solution of the problem is ensured.
Furthermore, when the problem satisfies certain conditions, the whole generated sequence converges to the solution.

The numerical experiments showed that our algorithm runs stably and lightly even if the constraint set is too complex for the existing method to run quickly.
Therefore, the proposed algorithm is useful for solving complicated constrained quasiconvex optimization problems.

\begin{comment}
As we saw in this paper, we can express many kinds of constraint sets, such as the intersection of simpler sets, the generalized convex feasible set, and so on, as fixed point sets of nonexpansive mappings.
However, there exists a more flexible, generalized class of the mapping called the quasi-nonexpansive mapping.
A fixed point set of this mapping can express a more complicated constraint set such as the slice of a convex function.
We discussed under the assumption that the constraint set is expressed by a nonexpansive mapping in this paper.
For deeper studies, an algorithm and its analyses under the assumption that the constraint set is expressed by a quasi-nonexpansive mapping should be considered and discussed for wider applications.
\end{comment}

%%%%%%%%%%%%%%%%%%%%%%%%%%%%%%%%%%%%%%%%%%%%%%%%%%%%%%%%%%%%%%%%%%%%%%%%%%%%%%%
%%% ACKNOWLEDGEMENTS
%%%%%%%%%%%%%%%%%%%%%%%%%%%%%%%%%%%%%%%%%%%%%%%%%%%%%%%%%%%%%%%%%%%%%%%%%%%%%%%
\section*{Acknowledgements}
We are sincerely grateful to the editor, Steffen Rebennack, and the three anonymous reviewers for helping us improve the original manuscript.
This work was supported by the Japan Society for the Promotion of Science (JSPS KAKENHI Grant Numbers JP17J09220, JP18K11184).

%%%%%%%%%%%%%%%%%%%%%%%%%%%%%%%%%%%%%%%%%%%%%%%%%%%%%%%%%%%%%%%%%%%%%%%%%%%%%%%
%%% REFERENCES
%%%%%%%%%%%%%%%%%%%%%%%%%%%%%%%%%%%%%%%%%%%%%%%%%%%%%%%%%%%%%%%%%%%%%%%%%%%%%%%
%\bibliographystyle{apalike}
%\bibliography{biblio}
\nocite{hardy1988,nocedal2006,dyer1980}

\ifextended
\appendix\newpage
%%%%%%%%%%%%%%%%%%%%%%%%%%%%%%%%%%%%%%%%%%%%%%%%%%%%%%%%%%%%%%%%%%%%%%%%%%%%%%%
%%% APPENDIX: BOUNDEDNESS
%%%%%%%%%%%%%%%%%%%%%%%%%%%%%%%%%%%%%%%%%%%%%%%%%%%%%%%%%%%%%%%%%%%%%%%%%%%%%%%
\section{Sufficient condition for ensuring Assumption~(A\ref{assum:bounded})}\label{appendix:loft}
Here, we present a sufficient condition for ensuring the boundedness of the sequence generated by Algorithm~\ref{alg:main} without supposing the boundedness of the set $D$.
Firstly, we define coerciveness, which is used as a condition for the objective functional to ensure the boundedness of the generated sequence.
\begin{dfn}[{\cite[Definition~11.11, Proposition~11.12]{bc}}]
  Let $H$ be a real Hilbert space, and let $f$ be a functional on $H$.
  We call $f$ is \emph{coercive} if
  \begin{align*}
    \lim_{\norm{x}\to\infty}f(x)=\infty.
  \end{align*}
  Furthermore, let $\lev_{\le \alpha}f$ denote the trench of $f$, i.e., $\lev_{\le \alpha}f:=\{x\in\Hsp:f(x)\le\alpha\}$, for any real number $\alpha\in\real$.
  Then, coerciveness is equivalent to the boundedness of all trenches of $f$.
\end{dfn}

Let us define the diameter of a set $C$ as $\diam(C):=\sup\{\norm{u-v}:u,v\in C\}$ for the proof described later.
The following proposition shows the necessary condition for ensuring the boundedness of the sequence generated by Algorithm~\ref{alg:main}.
\begin{prop}\label{prop:seqbdd}
  Let $\{x_k\}\subset\Hsp$ be a sequence generated by Algorithm~\ref{alg:main}.
  Suppose that Assumption~\ref{assum:basic} holds and there exists a number $k_0\in\nat$ such that $v_k<1$ for all $k\ge k_0$.
  Assume that $f$ is coercive.
  If one of the following holds,
  \begin{enumerate}[(i)]
    \item Assumption~(A\ref{assum:holder}) holds, \label{case:g:c1}
    \item the whole space $\Hsp$ is an $N$-dimensional Euclidean space $\real^n$, \label{case:g:c2}
  \end{enumerate}
  then Assumption~(A\ref{assum:bounded}) is satisfied.
\end{prop}
\begin{proof}
  Let us prove this proposition by dividing it into case~(\ref{case:g:c1}) and (\ref{case:g:c2}).
  First, suppose that case~(\ref{case:g:c1}) holds.
  We will proceed by way of contradiction and suppose that the sequence $\{x_k\}$ is unbounded.
  Then, there exists a subsequence $\{x_{k_i}\}$ of the sequence $\{x_k\}$ such that $\lim_{i\to\infty}\norm{x_{k_i}}=\infty$.
  The coercivity of $f$ implies that $\lim_{i\to\infty}f(x_{k_i})=\infty$.
  Fix $x^\star\in X^\star$ arbitrarily.
  Assumption~(A\ref{assum:holder}) guarantees that
  \begin{align*}
    \abs{f(z)-f_\star}
    \le L\norm{z-x^\star}^\beta
  \end{align*}
  for all $z\in \cl(\lev_{<f(x_i)}f$ and for any $i\in\nat$.
  Since $\lim_{i\to\infty}f(x_{k_i})=\infty$, we have
  \begin{align*}
    \abs{f(z)-f_\star}
    \le L\norm{z-x^\star}^\beta
  \end{align*}
  for all $z\in\Hsp$.
  Considering the point $z$ appearing in the above inequality to be limited in the set $x^\star+\Ball$, we obtain
  \begin{align*}
    f(z)
    & \le f_\star + L\norm{z-x^\star}^\beta\\
    & \le f_\star + L
  \end{align*}
  for all $z\in x^\star+\Ball$.
  Set $\delta$ to be the right side of the above inequality, i.e., $\delta:=f_\star+L$.
  Then, the set $x^\star + \Ball$ is obviously a subset of the bounded trench $\lev_{\le \delta}f$.

  From the assumption of this proposition, there exists a number $k_0\in\nat$ such that $v_k<1$ for all $k\ge k_0$.
  For each $k\ge k_0$, let us consider the two separate cases: the case where the point $x_k$ belongs to the trench $\lev_{\le \delta}f$, and its negation.
  First, let us consider the positive case; i.e., $x_k\in\lev_{\le \delta}f$ for $k\ge k_0$.
  The nonexpansivity of $P_D$ and $T$ and the fact that $x^\star\in\Fix(P_D)\cap\Fix(T)$ ensure that
  \begin{align*}
    \norm{x_{k+1}-x^\star}
    &=\norm{P_D(\alpha_k x_k + (1-\alpha_k)T(x_k-v_kg_k))-x^\star} \\
    &\le\norm{\alpha_k x_k+(1-\alpha_k)T(x_k-v_kg_k)-x^\star} \\
    &\le\alpha_k\norm{x_k-x^\star}+(1-\alpha_k)\norm{T(x_k-v_kg_k)-x^\star} \\
    &\le\alpha_k\norm{x_k-x^\star}+(1-\alpha_k)\norm{x_k-v_kg_k-x^\star} \\
    &\le\norm{x_k-x^\star}+(1-\alpha_k)v_k
  \end{align*}
  for any $k\ge k_0$ such that $x_k\in\lev_{\le \delta}f$.
  Here, both $x_k$ and $x^\star$ belong to the bounded trench $\lev_{\le \delta}f$, and both $1-\alpha$ and $v_k$ are less than or equal to $1$.
  Hence,
  \begin{align*}
    \norm{x_{k+1}-x^\star}
    &\le\diam(\lev_{\le\delta}f)+1<\infty
  \end{align*}
  holds for any $k\ge k_0$ such that $x_k\in\lev_{\le \delta}f$.
  Next, let us consider the negative case; i.e., $x_k\not\in\lev_{\le \delta}f$ for $k\ge k_0$.
  The nonexpansivity of $P_D$ and $T$ and the fact that $x^\star\in\Fix(P_D)\cap\Fix(T)$ ensure that
  \begin{align*}
    \norm{x_{k+1}-x^\star}
    &=\norm{P_D(\alpha_k x_k + (1-\alpha_k)T(x_k-v_kg_k))-x^\star} \\
    &\le\norm{\alpha_k x_k+(1-\alpha_k)T(x_k-v_kg_k)-x^\star} \\
    &\le\alpha_k\norm{x_k-x^\star}+(1-\alpha_k)\norm{T(x_k-v_kg_k)-x^\star} \\
    &\le\alpha_k\norm{x_k-x^\star}+(1-\alpha_k)\norm{x_k-v_kg_k-x^\star}
  \end{align*}
  for any $k\ge k_0$ such that $x_k\not\in\lev_{\le \delta}f$.
  Let us consider the right term of the right side of the above inequality.
  Its squared value is bounded from above as follows:
  \begin{align*}
    \norm{x_k-v_kg_k-x^\star}^2
    &=\norm{x_k-x^\star}^2-2v_k\ip{g_k,x_k-x^\star}+v_k\ip{g_k,v_kg_k} \\
    &=\norm{x_k-x^\star}^2-v_k\ip{g_k,x_k-x^\star}-v_k\ip{g_k,x_k-(x^\star+v_kg_k)}
  \end{align*}
  for any $k\ge k_0$ such that $x_k\not\in\lev_{\le \delta}f$.
  Here, $x_k\not\in\lev_{\le \delta}f$ implies that $f_\star\le f_\star+L=\delta<f(x_k)$ for $f\ge k_0$.
  Therefore, we have $x^\star\in\lev_{\le f(x_k)}f$ and $x^\star+v_kg_k\in x^\star+\Ball\subset\lev_{\le f(x_k)}f$ for any $k\ge k_0$ such that $x_k\not\in\lev_{\le \delta}f$.
  Since the definition of $g_k\in\partial^\star f(x_k)\cap\Sphere$ together with the preceding discussion implies that $\ip{g_k,x_k-x^\star}\ge 0$ and $\ip{g_k,x_k-(x^\star+v_kg_k)}\ge 0$, we have
  \begin{align*}
    \norm{x_k-v_kg_k-x^\star}
    &\le\norm{x_k-x^\star}
  \end{align*}
  for any $k\ge k_0$ such that $x_k\not\in\lev_{\le \delta}f$.
  Hence, we have
  \begin{align*}
    \norm{x_{k+1}-x^\star}
    \le \norm{x_k-x^\star}
  \end{align*}
  for any $k\ge k_0$ such that $x_k\not\in\lev_{\le \delta}f$.
  From the results for the both cases where $x_k\in\lev_{\le \delta}f$ or not, we have
  \begin{align*}
    \norm{x_k-x^\star}
    &\le\max\{\norm{x_1-x^\star}, \norm{x_2-x^\star},\ldots,\norm{x_{k_0}-x^\star}, \diam(\lev_{\le\delta}f)+1\} \\
    &<\infty
  \end{align*}
  for all $k\in\nat$.
  However, this contradicts the assumption that the sequence $\{x_k\}$ is unbounded.
  Therefore, we arrive at a contradiction and the boundedness of the sequence $\{x_k\}$ has been proved under this case.

  Now let us suppose that case~(\ref{case:g:c2}) holds.
  In an $N$-dimensional Euclidean space $\real^N$, every closed, bounded set is compact \cite[Problem~3.1.6]{wataru}.
  Therefore, the nonempty, compact set $x^\star+\Ball$ contains a point $\bar{x}$ such that $f(\bar{x})=\max_{x\in x^\star+\Ball}f(x)$ \cite[Theorem~2.5.7]{wataru}.
  Set $\delta:=f(\bar{x})\ge f_\star$.
  Then, the set $x^\star + \Ball$ is a subset of the bounded trench $\lev_{\le \delta}f=\lev_{\le f(\bar{x})}f$.

  From the assumption of this proposition, there exists a number $k_0\in\nat$ such that $v_k<1$ for all $k\ge k_0$.
  For each $k\ge k_0$, let us consider the two separate cases: the case where the point $x_k$ belongs to the trench $\lev_{\le \delta}f$, and its negation.
  First, let us consider the positive case; i.e., $x_k\in\lev_{\le \delta}f$ for $k\ge k_0$.
  The nonexpansivity of $P_D$ and $T$ and the fact that $x^\star\in\Fix(P_D)\cap\Fix(T)$ ensure that
  \begin{align*}
    \norm{x_{k+1}-x^\star}
    &=\norm{P_D(\alpha_k x_k + (1-\alpha_k)T(x_k-v_kg_k))-x^\star} \\
    &\le\norm{\alpha_k x_k+(1-\alpha_k)T(x_k-v_kg_k)-x^\star} \\
    &\le\alpha_k\norm{x_k-x^\star}+(1-\alpha_k)\norm{T(x_k-v_kg_k)-x^\star} \\
    &\le\alpha_k\norm{x_k-x^\star}+(1-\alpha_k)\norm{x_k-v_kg_k-x^\star} \\
    &\le\norm{x_k-x^\star}+(1-\alpha_k)v_k
  \end{align*}
  for any $k\ge k_0$ such that $x_k\in\lev_{\le \delta}f$.
  Here, both $x_k$ and $x^\star$ belong to the bounded trench $\lev_{\le \delta}f$, and both $1-\alpha$ and $v_k$ are less than or equal to $1$.
  Hence,
  \begin{align*}
    \norm{x_{k+1}-x^\star}
    &\le\diam(\lev_{\le\delta}f)+1<\infty
  \end{align*}
  holds for any $k\ge k_0$ such that $x_k\in\lev_{\le \delta}f$.
  Next, let us consider the negative case; i.e., $x_k\not\in\lev_{\le \delta}f$ for $k\ge k_0$.
  The nonexpansivity of $P_D$ and $T$ and the fact that $x^\star\in\Fix(P_D)\cap\Fix(T)$ ensure that
  \begin{align*}
    \norm{x_{k+1}-x^\star}
    &=\norm{P_D(\alpha_k x_k + (1-\alpha_k)T(x_k-v_kg_k))-x^\star} \\
    &\le\norm{\alpha_k x_k+(1-\alpha_k)T(x_k-v_kg_k)-x^\star} \\
    &\le\alpha_k\norm{x_k-x^\star}+(1-\alpha_k)\norm{T(x_k-v_kg_k)-x^\star} \\
    &\le\alpha_k\norm{x_k-x^\star}+(1-\alpha_k)\norm{x_k-v_kg_k-x^\star}
  \end{align*}
  for any $k\ge k_0$ such that $x_k\not\in\lev_{\le \delta}f$.
  Let us consider the right term of the right side of the above inequality.
  Its squared value is bounded from above as follows:
  \begin{align*}
    \norm{x_k-v_kg_k-x^\star}^2
    &=\norm{x_k-x^\star}^2-2v_k\ip{g_k,x_k-x^\star}+v_k\ip{g_k,v_kg_k} \\
    &=\norm{x_k-x^\star}^2-v_k\ip{g_k,x_k-x^\star}-v_k\ip{g_k,x_k-(x^\star+v_kg_k)}
  \end{align*}
  for any $k\ge k_0$ such that $x_k\not\in\lev_{\le \delta}f$.
  Here, $x_k\not\in\lev_{\le \delta}f$ implies that $f_\star\le f_\star+L=\delta<f(x_k)$ for $f\ge k_0$.
  Therefore, we have $x^\star\in\lev_{\le f(x_k)}f$ and $x^\star+v_kg_k\in x^\star+\Ball\subset\lev_{\le f(x_k)}f$ for any $k\ge k_0$ such that $x_k\not\in\lev_{\le \delta}f$.
  Since the definition of $g_k\in\partial^\star f(x_k)\cap\Sphere$ together with the preceding discussion implies that $\ip{g_k,x_k-x^\star}\ge 0$ and $\ip{g_k,x_k-(x^\star+v_kg_k)}\ge 0$, we have
  \begin{align*}
    \norm{x_k-v_kg_k-x^\star}
    &\le\norm{x_k-x^\star}
  \end{align*}
  for any $k\ge k_0$ such that $x_k\not\in\lev_{\le \delta}f$.
  Hence, we have
  \begin{align*}
    \norm{x_{k+1}-x^\star}
    \le \norm{x_k-x^\star}
  \end{align*}
  for any $k\ge k_0$ such that $x_k\not\in\lev_{\le \delta}f$.
  From the results for the both cases where $x_k\in\lev_{\le \delta}f$ or not, we have
  \begin{align*}
    \norm{x_k-x^\star}
    &\le\max\{\norm{x_1-x^\star}, \norm{x_2-x^\star},\ldots,\norm{x_{k_0}-x^\star}, \diam(\lev_{\le\delta}f)+1\} \\
    &<\infty
  \end{align*}
  for all $k\in\nat$.
  This implies that the sequence $\{x_k\}$ is bounded, and this completes the proof.
\end{proof}
The assumption on the step-size, i.e., the existence of a number $k_0$ such that $v_k<1$ for all $k\ge k_0$, is satisfied whenever we adopt a diminishing step-size rule since the step-size sequence decreasingly converges to zero.
Furthermore, it is also satisfied when we adopt a constant step-size rule with a small enough constant step-size $v<1$.
Overall, the coerciveness of the objective functional $f$ (and the smallness of the step-sizes) is a sufficient condition for ensuring the boundedness of the generated sequence in the convergence analyses of constant or diminishing step-size rules.

%%%%%%%%%%%%%%%%%%%%%%%%%%%%%%%%%%%%%%%%%%%%%%%%%%%%%%%%%%%%%%%%%%%%%%%%%%%%%%%
%%% APPENDIX: MISCELLANEOUS PROOFS
%%%%%%%%%%%%%%%%%%%%%%%%%%%%%%%%%%%%%%%%%%%%%%%%%%%%%%%%%%%%%%%%%%%%%%%%%%%%%%%
\section{Proofs of Proposition~\ref{prop:konnov} and some examples}
\begin{proof}[Proof of Proposition~\ref{prop:konnov}]
Fix $g\in\partial^\star f(x)\cap\Sphere$ arbitrarily.
The continuity and quasiconvexity of $f$ imply its level set $\lev_{<f(x)}f$ is open and convex.
This $\lev_{<f(x)}f$ is not an empty set, since it has at least the point $x^\star$.
The continuity of $f$ also ensures $\bd(\lev_{<f(x)}f)\neq\emptyset$.

Set $r:=\inf\{\norm{x^\star-u}:u\in\bd(\lev_{<f(x)}f)\}$.
Then, there exists a sequence $\{u_k\}\subset\bd(\lev_{<f(x)}f)$ such that $\norm{x^\star-u_k}\le r+1/k$ for all $k\in\nat$.
The openness of $\lev_{<f(x)}f$ implies that it is a distinct set from its boundary, i.e., $f(x)\le f(u)$ for any $u\in\bd(\lev_{<f(x)}f)$.
Hence,
\begin{align*}
f(x)-f_\star
&\le f(u_k)-f_\star\\
&\le L\norm{u_k-x^\star}^\beta\\
&<L\left(r+\frac{1}{k}\right)^\beta\quad(k\in\nat).
\end{align*}
It follows that
\begin{align}
f(x)-f_\star
\le Lr^\beta.\label{eqn:1}
\end{align}

From the definition of $r$, the open ball with center $x^\star$ and radius $r$ is contained inside $\lev_{<f(x)}f$.
Therefore, $x^\star+(1-1/k)rg\in\lev_{<f(x)}f$ holds for any $k\in\nat$, and we have
\begin{align*}
\left(1-\frac{1}{k}\right)r-\ip{g,x-x^\star}
&=\left(1-\frac{1}{k}\right)r\norm{g}^2-\ip{g,x-x^\star}\\
&=\ip{g,x^\star+\left(1-\frac{1}{k}\right)rg-x}\le 0\quad(k\in\nat).
\end{align*}
This implies that $r\le\ip{g,x-x^\star}$.
Applying this inequality to inequality~\eqref{eqn:1} gives $f(x)-f_\star\le L\ip{g,x-x^\star}^\beta$.
This completes the proof.
\end{proof}

\begin{proof}[Proof of Example~\ref{example:2}]
The fixed point set $\Fix(T)=\Fix(\Id)$ is obviously the whole space $\Hsp$.
Therefore, the minimum value of $f$ is $0$ and its minimizer is only the origin.
For any $z\in\Hsp$, we have
\begin{align*}
\abs{f(z)-f(x^\star)}
&=\min\{\norm{z},\alpha\}\\
&\le\norm{z-x^\star}.
\end{align*}
This implies that $f$ satisfies the H\"older condition with degree $1$ at its minimizer $x^\star$ on the whole space $\Hsp$.
Expanding $\norm{x_{k+1}}^2=\norm{x_k-(v_k/2)g_k}^2$ and using Proposition~\ref{prop:konnov} with the above result, we have
\begin{align*}
\norm{x_{k+1}}^2
&=\norm{x_k}^2-v_k\ip{g_k,x_k-x^\star}+\left(\frac{v_k}{2}\right)^2\\
&\le\norm{x_k}^2-v_k\min\{\norm{x_k},\alpha\}+\left(\frac{v_k}{2}\right)^2\quad(k\in\nat).
\end{align*}
When the index $k\in\nat$ satisfies $\norm{x_k}\le\alpha$,
\begin{align*}
\norm{x_{k+1}}^2
&\le\norm{x_k}^2-v_k\norm{x_k}+\left(\frac{v_k}{2}\right)^2\\
&\le\frac{5}{4}\alpha^2
\end{align*}
holds.
The nonnegativeness of both sides of the above inequality ensures that $\norm{x_{k+1}}$ is bounded from above by $\sqrt{5}\alpha/2$ for any $k\in\nat$ satisfying $\norm{x_k}\le\alpha$.
In the opposite case, i.e. if the index $k\in\nat$ satisfies $\alpha<\norm{x_k}$,
\begin{align*}
\norm{x_{k+1}}^2
&\le\norm{x_k}^2-v_k\alpha+\left(\frac{v_k}{2}\right)^2\\
&=\norm{x_k}^2-v_k\left(\alpha-\frac{v_k}{4}\right)
\end{align*}
holds.
Therefore, we have $\norm{x_{k+1}}\le\norm{x_k}$ for any $k\in\nat$ satisfying $\alpha<\norm{x_k}$.
Combining the conclusions of both cases, we can see that the sequence $\{\norm{x_k}\}$ is bounded from above by $\max\{\norm{x_1},\sqrt{5}\alpha/2\}$.
This completes the proof.
\end{proof}

\begin{proof}[Proof of Example~\ref{example:3}]
The generated sequence $\{x_k\}$ is obviously bounded, since it is contained inside $[0,M]^n$.
Fix $x^\star\in X^\star$ arbitrarily.
The assumptions of this example imply that the feasible set $\Fix(T)\cap D$ has a point $u\in(0,\infty)^n$.
Hence, the functional value of this point $f(u)$ is strictly less than $f(v)=0$ for any $v\in\bd([0,\infty)^n)$.
This means that $X^\star\subset(0,\infty)^n$ holds; that is, any optimum $x^\star\in X^\star$ belongs to the interior of $[0,\infty)^n$.
Therefore, there exists some $\delta>0$ such that the closed ball with center $x^\star$ and radius $2\delta$ is contained inside $[0,\infty)^n$.
Let us discuss the satisfiability of the H\"older condition on the set $[\delta,\infty)^n$ and on its complement $\overline{[\delta,\infty)^n}$ separately (see Figure~\ref{fig:ballfig} for the relation between the point $x^\star$ and these two disjoint sets).
\begin{figure}[htbp]
  \centering
  \includegraphics[width=5truecm]{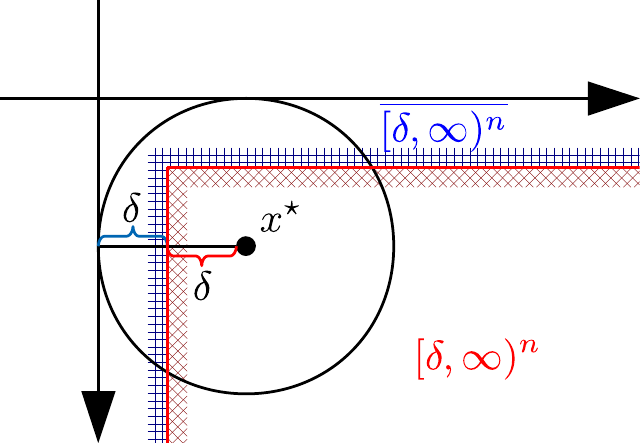}
  \caption{Relation between the point $x^\star$, the set $\overline{[\delta,\infty)^n}$, and its complement}\label{fig:ballfig}
\end{figure}
First, fix $z\in[\delta,\infty)^n$ arbitrarily.
Denoting the element-wise (Hadamard) division of $x$ and $y$ by $x\oslash y$, the gradient of the functional $f$ can be written as
\begin{align*}
\nabla f(x)=\frac{-a_0\prod_{j=1}^nx_j^{a_j}}{(\ip{c,x}+c_0)^2}c+\frac{-a_0\prod_{j=1}^nx_j^{a_j}}{\ip{c,x}+c_0}(a\oslash x)\quad(x\in[\delta,\infty)^n).
\end{align*}
Letting $\tilde{c}$ be a vector whose elements are each the minimum of $c$, an upper bound of the norms of the gradients can be evaluated by using the theorem of arithmetic and geometric means \cite[Inequality~(2.5.2)]{hardy1988}, as follows:
\begin{align*}
\norm{\nabla f(x)}
&\le\frac{a_0\prod_{j=1}^nx_j^{a_j}}{(\ip{c,x}+c_0)^2}\norm{c}+\frac{a_0\prod_{j=1}^nx_j^{a_j}}{\ip{c,x}+c_0}\norm{a\oslash x}\\
&\le\frac{a_0}{\norm{\tilde{c}}}\left(\frac{1}{\ip{c,x}+c_0}\norm{c}+\norm{a\oslash x}\right)\\
&\le\frac{a_0}{\norm{\tilde{c}}}\left(\frac{1}{\delta\norm{\tilde{c}}+c_0}\norm{c}+\frac{\norm{a}}{\delta}\right)<\infty\quad(x\in[\delta,\infty)^n).
\end{align*}
This implies that the image $(\norm{\nabla f(\cdot)})([\delta,\infty)^n)$ is bounded.
The mean value theorem \cite[Inequality~(A.55)]{nocedal2006} ensures the existence of some $\alpha\in(0,1)$ such that
\begin{align*}
f(z)=f(x^\star)+\ip{\nabla f((1-\alpha)x^\star+\alpha z), z-x}.
\end{align*}
The convexity of $[\delta,\infty)$ and the boundedness of $\norm{\nabla f(\cdot)}$ on $[\delta,\infty)^n$ imply that $\norm{\nabla f((1-\alpha)x^\star+\alpha z)}\le(a_0/\norm{\tilde{c}})(\norm{c}/(\delta\norm{\tilde{c}}+c_0)+\norm{a}/\delta)$.
Therefore, the Cauchy-Schwarz inequality gives us the desired inequality as follows:
\begin{align*}
\abs{f(z)-f(x^\star)}
&=\abs{\ip{\nabla f((1-\alpha)x^\star+\alpha z), z-x}}\\
&\le\norm{\nabla f((1-\alpha)x^\star+\alpha z)}\norm{z-x}\\
&\le\frac{a_0}{\norm{\tilde{c}}}\left(\frac{1}{\delta\norm{\tilde{c}}+c_0}\norm{c}+\frac{\norm{a}}{\delta}\right)\norm{z-x}.
\end{align*}
This implies that the functional $f$ satisfies the H\"older condition with degree $1$ at the point $x^\star$ on the set $[\delta,\infty)^n$.

Next, fix $\bar{z}\in\overline{[\delta,\infty)^n}$ arbitrarily.
The closed ball with center $x^\star$ and radius $\delta$ is contained inside $[\delta,\infty)^n$.
Hence, $\norm{x^\star-z}\ge\delta$ holds.
From the definition of $f$, the maximum value of $f$ is $0$ and its range is less than or equal to $0$.
Therefore,
\begin{align*}
\abs{f(z)-f(x^\star)}
\le \frac{f(x^\star)}{\delta}\norm{x^\star-z}
\end{align*}
holds.
Hence, letting $L:=\max\{(a_0/\norm{\tilde{c}})(\norm{c}/(\delta\norm{\tilde{c}}+c_0)+\norm{a}/\delta),f(x^\star/\delta)\}$, it is clear that $f$ satisfies the H\"older condition with degree $1$ at the point $x^\star$ on the set $\real^n$.
This completes the proof.
\end{proof}

%%%%%%%%%%%%%%%%%%%%%%%%%%%%%%%%%%%%%%%%%%%%%%%%%%%%%%%%%%%%%%%%%%%%%%%%%%%%%%%
%%% APPENDIX: PROOFS OF LEMMAS
%%%%%%%%%%%%%%%%%%%%%%%%%%%%%%%%%%%%%%%%%%%%%%%%%%%%%%%%%%%%%%%%%%%%%%%%%%%%%%%
\section{Proofs of Lemmas~\ref{lem:fun} and \ref{lem:dist}}\label{appendix:lemmasproofs}
\begin{proof}[Proof of Lemma~\ref{lem:fun}]
  Fix $x^\star\in X^\star$ and $k\in\nat$ arbitrarily.
  The convexity of $\norm{\cdot}^2$ and the nonexpansivity of $P_D$ and $T$ ensure that
  \begin{align}
  \norm{x_{k+1}-x^\star}^2
  &=\norm{P_D(\alpha_kx_k+(1-\alpha_k)T(x_k-v_kg_k))-P_D(x^\star)}^2\nonumber\\
  &\le\norm{\alpha_kx_k+(1-\alpha_k)T(x_k-v_kg_k)-x^\star}^2\nonumber\\
  &\le\alpha_k\norm{x_k-x^\star}^2+(1-\alpha_k)\norm{T(x_k-v_kg_k)-T(x^\star)}^2\nonumber\\
  &\le\alpha_k\norm{x_k-x^\star}^2+(1-\alpha_k)\norm{x_k-x^\star-v_kg_k}^2\nonumber\\
  &=\norm{x_k-x^\star}^2-2v_k(1-\alpha_k)\ip{g_k,x_k-x^\star}+(1-\alpha_k)v_k^2.\label{eqn:2}
  \end{align}
  On the other hand, Assumption~(A\ref{assum:holder}) and Proposition~\ref{prop:konnov} ensure that
  \begin{align*}
  \left(\frac{f(x_k)-f_\star}{L}\right)^{\frac{1}{\beta}}
  \le\ip{g_k,x_k-x^\star}.
  \end{align*}
  Applying this inequality to inequality~\eqref{eqn:2}, we obtain
  \begin{align*}
  \norm{x_{k+1}-x^\star}^2
  &\le\norm{x_k-x^\star}^2-2v_k(1-\alpha_k)\left(\frac{f(x_k)-f_\star}{L}\right)^{\frac{1}{\beta}}+(1-\alpha_k)v_k^2.
  \end{align*}
  This completes the proof.
\end{proof}

\begin{proof}[Proof of Lemma~\ref{lem:dist}]
  Fix $x\in\Fix(T)\cap D$ and $k\in\nat$ arbitrarily, and set $M_1:=\sup_{k\in\nat}(2\abs{\ip{g,x-T(x_k-v_kg_k)}})$.
  The boundedness of $\{x_k\}$ and $\{v_k\}$ ensures that there exist $M_2,M_3\in\real$ such that
  \begin{align*}
  \norm{x_j}\le M_2,\quad
  v_j\le M_3\quad(j\in\nat).
  \end{align*}
  Therefore,
  \begin{align*}
  \norm{x-T(x_j-v_jg_j)}
  &\le\norm{x}+\norm{x_j}+v_j\\
  &\le\norm{x}+M_2+M_3<\infty\quad(j\in\nat).
  \end{align*}
  The Cauchy-Schwarz inequality, together with this boundedness of the real sequence $\{\norm{x-T(x_k-v_kg_k)}\}$, indicates that $M_1\le\norm{x}+M_2+M_3<\infty$.
  
  From the convexity of $\norm{\cdot}^2$, we have
  \begin{align}
  \norm{x_{k+1}-x}^2
  &=\norm{P_D(\alpha_kx_k+(1-\alpha_k)T(x_k-v_kg_k))-P_D(x)}^2\nonumber\\
  &\le\norm{\alpha_k(x_k-x)+(1-\alpha_k)(T(x_k-v_kg_k)-x)}^2\nonumber\\
  &\le\alpha_k\norm{x_k-x}^2+(1-\alpha_k)\norm{T(x_k-v_kg_k)-x}^2.\label{eqn:3}
  \end{align}
  Let us consider the term $\norm{T(x_k-v_kg_k)-x}^2$.
  Using the firm nonexpansivity of $T$, we expand this term into
  \begin{align*}
  &\norm{T(x_k-v_kg_k)-x}^2\\
  &\le\norm{x_k-v_kg_k-x}^2-\norm{(\Id-T)(x_k-v_kg_k)-(\Id-T)(x)}^2\\
  &=\norm{x_k-x}^2-2v_k\ip{g_k,x_k-x}+v_k^2\\
  &\quad-\norm{x_k-T(x_k-v_kg_k)}^2+2v_k\ip{g_k,x_k-T(x_k-v_kg_k)}-v_k^2\\
  &=\norm{x_k-x}^2-\norm{x_k-T(x_k-v_kg_k)}^2+2v_k\ip{g_k,x-T(x_k-v_kg_k)}.
  \end{align*}
  In view of the definition of $M_1$, the set $\{2v_k\ip{g_k,x-T(x_k-v_kg_k)}:k\in\nat\}$ is bounded from above by it.
  Therefore, we obtain
  \begin{align*}
  \norm{T(x_k-v_kg_k)-x}^2
  \le\norm{x_k-x}^2-\norm{x_k-T(x_k-v_kg_k)}^2+v_kM_1.
  \end{align*}
  Applying this inequality to inequality~\eqref{eqn:3} yields the desired inequality:
  \begin{align*}
  &\norm{x_{k+1}-x}^2\\
  &\le\alpha_k\norm{x_k-x}^2+(1-\alpha_k)\left(\norm{x_k-x}^2-\norm{x_k-T(x_k-v_kg_k)}^2+v_kM_1\right)\\
  &\le\norm{x_k-x}^2-(1-\alpha_k)\norm{x_k-T(x_k-v_kg_k)}^2+v_kM_1
  \end{align*}
  This completes the proof.
\end{proof}

%%%%%%%%%%%%%%%%%%%%%%%%%%%%%%%%%%%%%%%%%%%%%%%%%%%%%%%%%%%%%%%%%%%%%%%%%%%%%%%
%%% APPENDIX: PROOF OF CONSTANT STEP-SIZE THEOREM
%%%%%%%%%%%%%%%%%%%%%%%%%%%%%%%%%%%%%%%%%%%%%%%%%%%%%%%%%%%%%%%%%%%%%%%%%%%%%%%
\section{Proof of Theorem~\ref{thm:constant}}\label{appendix:constantproof}
\begin{proof}
We prove each inequality in order.
First, we consider whether the inequality
\begin{align}
\liminf_{k\to\infty}f(x_k)\le f_\star+L\left(\frac{v}{2}\right)^\beta\label{eqn:4}
\end{align}
holds.
We will proceed by way of contradiction. Suppose that the inequality does not hold; i.e.,
\begin{align*}
f_\star+L\left(\frac{v}{2}\right)^\beta<\liminf_{k\to\infty}f(x_k).
\end{align*}
The left-hand side of this inequality is strictly less than the right-hand side.
Hence, with the positivity of $L$, we can choose a positive $\delta_1$ such that
\begin{align*}
f_\star+L\left(\frac{v}{2}+\delta_1\right)^\beta<\liminf_{k\to\infty}f(x_k)
\end{align*}
holds.
The property of the limit inferior guarantees that there exists $k_0\in\nat$ such that
\begin{align}
f_\star+L\left(\frac{v}{2}+\delta_1\right)^\beta<f(x_k)\quad(k\ge k_0).\label{eqn:5}
\end{align}
Obviously, this implies that $f_\star<f(x_k)$ for any $k\ge k_0$.
Therefore, all assumptions of Lemma~\ref{lem:fun} are satisfied when $k\ge k_0$, and the following inequality holds for some $x^\star\in X^\star\neq\emptyset$.
\begin{align*}
&\norm{x_{k+1}-x^\star}^2\\
&\le\norm{x_k-x^\star}^2-2v(1-\alpha_k)\left(\frac{f(x_k)-f_\star}{L}\right)^\frac{1}{\beta}+(1-\alpha_k)v^2\quad(k\ge k_0).
\end{align*}
Applying inequality~\eqref{eqn:5} to the above inequality, we have
\begin{align}
\norm{x_{k+1}-x^\star}^2
&\le\norm{x_k-x^\star}^2-2v\delta_1(1-\alpha_k)\nonumber\\
&\le\norm{x_{k_0}-x^\star}^2-2v\delta_1\sum_{j=k_0}^k(1-\alpha_k)\quad(k\ge k_0).\label{eqn:6}
\end{align}
From Assumption~(A\ref{assum:alpha}), $\limsup_{k\to\infty}\alpha_k<1$ holds.
Hence, a starting index $k_1\in\nat$ greater than $k_0$ exists such that the subsequence $\{\alpha_k\}_{k\ge k_1}$ is bounded above by some positive real that is strictly less than 1.
This means that inequality~\eqref{eqn:6} does not hold for large enough $k\ge k_1$, and we have arrived at a contradiction.
Therefore, inequality~\eqref{eqn:4} holds.

Next, let us prove the remaining part of this theorem, in other words, show that the inequality
\begin{align*}
\liminf_{k\to\infty}\norm{x_k-T(x_k)}^2\le Mv
\end{align*}
holds for some positive real $M>0$.
Fix $x\in\Fix(T)$ arbitrarily.
Lemma~\ref{lem:dist} guarantees the existence of a nonnegative real $M_1\ge 0$ such that
\begin{align}
\norm{x_{k+1}-x}^2
\le\norm{x_k-x}^2-(1-\alpha_k)\norm{x_k-T(x_k-vg_k)}^2+vM_1\quad(k\in\nat).\label{eqn:7}
\end{align}
In view of Assumption~(A\ref{assum:alpha}), $\liminf_{k\to\infty}(1-\alpha_k)$ is positive; i.e., it is not equal to zero.
We will again proceed by way of contradiction and suppose that
\begin{align}
\liminf_{k\to\infty}\norm{x_k-T(x_k-vg_k)}^2
\le\frac{2vM_1}{\liminf_{k\to\infty}(1-\alpha_k)}\label{eqn:8}
\end{align}
does not hold and
\begin{align*}
\frac{2vM_1}{\liminf_{k\to\infty}(1-\alpha_k)}
<\liminf_{k\to\infty}\norm{x_k-T(x_k-vg_k)}^2
\end{align*}
holds.
In the same ways choosing $\delta_1$ in the first part of this proof, we can find a positive $\delta_2>0$ that satisfies
\begin{align*}
\frac{2vM_1}{\liminf_{k\to\infty}(1-\alpha_k)}+\delta_2
<\liminf_{k\to\infty}\norm{x_k-T(x_k-vg_k)}^2.
\end{align*}
The property of the limit inferior guarantees that a positive number $k_2\in\nat$ exists such that
\begin{align*}
\frac{2vM_1}{\liminf_{k\to\infty}(1-\alpha_k)}+\delta_2
<\norm{x_k-T(x_k-vg_k)}^2\quad(k\ge k_2).
\end{align*}
Applying the above inequality to inequality~\eqref{eqn:7}, we obtain
\begin{align*}
&\norm{x_{k+1}-x}^2\\
&\le\norm{x_k-x}^2-(1-\alpha_k)\left(\frac{2vM_1}{\liminf_{k\to\infty}(1-\alpha_k)}+\delta_2\right)+vM_1\quad(k\ge k_2).
\end{align*}
The fundamental property of the limit inferior also ensures the existence of a number $k_3\in\nat$ larger than $k_2$ such that $\liminf_{k\to\infty}(1-\alpha_k)/2<(1-\alpha_k)$ for any $k\ge k_3$.
Therefore, we have
\begin{align*}
\norm{x_{k+1}-x}^2
&\le\norm{x_k-x}^2-\frac{\delta_2}{2}\liminf_{k\to\infty}(1-\alpha_k)\\
&\le\norm{x_{k_3}-x}^2-\frac{\delta_2}{2}(k-k_3+1)\liminf_{k\to\infty}(1-\alpha_k)\quad(k\ge k_3).
\end{align*}
Since the above inequality does not hold for large enough $k\ge k_3$, we have arrived at a contradiction.
Therefore, inequality~\eqref{eqn:8} holds.

With this inequality, let us evaluate the squared distance between an element of the generated sequence and its transformed point by the nonexpansive mapping $T$.
Using the triangle inequality and the nonexpansivity of $T$, we have
\begin{align*}
\norm{x_k-T(x_k)}^2
&\le\left(\norm{x_k-T(x_k-vg_k)}+\norm{T(x_k-vg_k)-T(x_k)}\right)^2\\
&\le\left(\norm{x_k-T(x_k-vg_k)}+v\right)^2\\
&=\norm{x_k-T(x_k-vg_k)}^2+2v\norm{x_k-T(x_k-vg_k)}+v^2\quad(k\in\nat).
\end{align*}
Now, since $x$ is a fixed point of the mapping $T$, we expand the second term of the above expression as $\norm{x_k-T(x_k-vg_k)}\le2\norm{x_k-x}+v$ for any $k\in\nat$.
Furthermore, Assumption~(A\ref{assum:bounded}) ensures the existence of a positive real $M_2>0$ bounding the set $\{\norm{x_k-x}:k\in\nat\}$ from above.
Hence, we finally obtain
\begin{align*}
\norm{x_k-T(x_k)}^2
&\le\norm{x_k-T(x_k-vg_k)}^2+4v(M_2+v)\quad(k\in\nat).
\end{align*}
Taking the limit inferior of both sides of the above inequality yields
\begin{align*}
\liminf_{k\to\infty}\norm{x_k-T(x_k)}^2
&\le\liminf_{k\to\infty}\norm{x_k-T(x_k-vg_k)}^2+4v(M_2+v).
\end{align*}
Set $M:=2(M_1/\liminf_{k\to\infty}(1-\alpha_k)+2(M_2+v))\in\real$.
Applying inequality~\eqref{eqn:8} to the above, we obtain the desired inequality as follows:
\begin{align*}
\liminf_{k\to\infty}\norm{x_k-T(x_k)}^2
&\le 2\left(\frac{M_1}{\liminf_{k\to\infty}(1-\alpha_k)}+2(M_2+v)\right)v\\
&\le Mv.
\end{align*}
This completes the proof.
\end{proof}

%%%%%%%%%%%%%%%%%%%%%%%%%%%%%%%%%%%%%%%%%%%%%%%%%%%%%%%%%%%%%%%%%%%%%%%%%%%%%%%
%%% APPENDIX: PROOF OF CONSTANT STEP-SIZE THEOREM
%%%%%%%%%%%%%%%%%%%%%%%%%%%%%%%%%%%%%%%%%%%%%%%%%%%%%%%%%%%%%%%%%%%%%%%%%%%%%%%
\section{Proof of Lemma~\ref{lem:dsr}}\label{appendix:dsrlemma}
\begin{proof}
  We will proceed by way of contradiction and suppose that the conclusion $\liminf_{k\to\infty}f(x_k)\le f_\star$ does not hold, that is, $\liminf_{k\to\infty}f(x_k)>f_\star$.
  There exists a positive number $\delta>0$ and an index $k_0\in\nat$ such that $f_\star+\delta<f(x_k)$ for all $k\ge k_0$.
  Furthermore, the assumption that the real sequence $\{v_k\}$ converges to zero guarantees the existence of an index $k_1\ge k_0$ such that $v_k<(\delta/L)^{1/\beta}$ for all $k\ge k_1$.
  Applying these two inequalities to Lemma~\ref{lem:fun}, we have
  \begin{align*}
  \norm{x_{k+1}-x^\star}^2
  &\le\norm{x_k-x^\star}^2-2v_k(1-\alpha_k)\left(\frac{f(x_k)-f_\star}{L}\right)^\frac{1}{\beta}+(1-\alpha_k)v_k^2\\
  &<\norm{x_k-x^\star}^2-v_k(1-\alpha_k)\left(\frac{\delta}{L}\right)^\frac{1}{\beta}\\
  &<\norm{x_{k_1}-x^\star}^2-\left(\frac{\delta}{L}\right)^\frac{1}{\beta}\sum_{n=k_1}^kv_k(1-\alpha_k)\quad(k\ge k_1).
  \end{align*}
  Assumption~(A\ref{assum:alpha}) guarantees the existence of a positive number $\underline{\alpha}$ and an index $k_2\ge k_1$ such that $\underline{\alpha}<1-\alpha_k$ for any $k\ge k_2$.
  This implies that the above inequality does not hold for a sufficiently large $k\ge k_2$; hence, we arrive at a contradiction.
  This completes the proof.
  \end{proof}

%%%%%%%%%%%%%%%%%%%%%%%%%%%%%%%%%%%%%%%%%%%%%%%%%%%%%%%%%%%%%%%%%%%%%%%%%%%%%%%
%%% APPENDIX: PROOF OF THE MAIN THEOREM IN EUCLIDEAN SPACE
%%%%%%%%%%%%%%%%%%%%%%%%%%%%%%%%%%%%%%%%%%%%%%%%%%%%%%%%%%%%%%%%%%%%%%%%%%%%%%%
\section{Proof of the Additional Statement of Theorem~\ref{thm:dsr}}\label{proof:thm:dsr}
\begin{proof}
Let us prove the additional statement of this theorem.
The following proof is played under the assumption that the solution $x^\star\in X^\star$ is unique and the whole space is an $N$-dimensional Euclidean space, i.e., $\Hsp=\real^N$.

The existence of a subsequence $\{x_{k_i}\}$ that weakly converges to a unique solution $x^\star$ is obtained from Theorem~\ref{thm:dsr}.
In Euclidean space, weak convergence coincides with strong convergence.
Therefore, the sequence $\{x_{k_i}\}$ converges to a unique $x^\star$.
If some number $k_0\in\nat$ exists such that $\norm{x_{k+1}-x^\star}\le\norm{x_k-x^\star}$ for all $k\ge k_0$, the whole sequence $\{x_k\}$ converges to a point in $X^\star$.
Let us consider the opposite case.
Let $\{k_i\}\subset\nat$ be the sequence of all indexes satisfying $\norm{x_{k_i}-x^\star}<\norm{x_{k_i+1}-x^\star}$ (and $k_i<k_{i+1}$ for any $i\in\nat$).
According to the assumption, this sequence is infinite.
The sequence $\{x_{k_i}\}$ is now bounded, and this implies that it has a subsequence converging to a unique optimum $x^\star\in X^\star$.
The above discussion ensures that any converging subsequence of $\{x_{k_i}\}$ converges to a unique optimum $x^\star\in X^\star$.
This further implies that $\{x_{k_i}\}$ also converges to this optimum $x^\star\in X^\star$.
From the assumed settings, $\norm{x_{j+1}-x^\star}\le\norm{x_j-x^\star}$ holds for any index $j\in\nat$ that does not belong to the set $\{k_i\}$.
The convergence of the sequence $\{x_{k_i}\}$ means that, for any $\epsilon>0$, there exists an index $\hat{i}\in\nat$ such that $\norm{x_{k_{\hat{i}}}-x^\star}<\epsilon$.
Furthermore, for any index $k\ge k_{\hat{i}}$ that does not belong to the set $\{k_i\}$, $\norm{x_{k_{\hat{i}}}-x^\star}\le\norm{x_{k_i}-x^\star}<\epsilon$ also holds for $i:=\max\{k_i:k_i\le k\}\ge k_{\hat{i}}$.
This implies that the whole sequence $\{x_k\}$ converges to a unique optimum $x^\star\in X^\star$.
This completes the proof.
\end{proof}

%%%%%%%%%%%%%%%%%%%%%%%%%%%%%%%%%%%%%%%%%%%%%%%%%%%%%%%%%%%%%%%%%%%%%%%%%%%%%%%
%%% APPENDIX: EFFICIENCY
%%%%%%%%%%%%%%%%%%%%%%%%%%%%%%%%%%%%%%%%%%%%%%%%%%%%%%%%%%%%%%%%%%%%%%%%%%%%%%%
\section{Efficiency}\label{sec:finiteconvergence}
Here, we discuss the rate of convergence of Algorithm~\ref{alg:main} in terms of the value of the objective functional and the distance to the fixed point set.
Furthermore, we discuss a sufficient condition to obtain finite convergence to some solution.

Let us start with the convergence rate analysis in terms of the objective function.
Here, we will use the following concepts originally introduced in \cite{yaohua2015,kiwiel2001}.
\begin{dfn}[{\cite[Section~5]{yaohua2015}, \cite[Section~6]{kiwiel2001}}]\label{M:15}
  Define the following notations:
  \begin{enumerate}[(i)]
    \item $x^\star_k:\in\argmin_{x\in\{x_1,x_2,\ldots,x_k\}}f(x)$,
    \item $r_k:=\sup\{r>0:x^\star_k+r\Ball\subset\lev_{<f(x^\star_k)}f\}$.
    % TODO: Fix x^\star_k -> x^\star and its description!!!
  \end{enumerate}
\end{dfn}
$x^\star_k$ expresses the best solution acquired until the $k$-th iteration, and $r_k$ expresses the distance between the level set of $x^\star_k$ and the optimal solution.
The difference between the above definitions and the original ones is in considering the possibility that the generated sequence may be out of the fixed point set, in other words, the constraint set.

We will use the following propositions.
\begin{prop}\label{prop:isom}
  Let $C\subset \Hsp$ be a nonempty, convex set, and suppose that $x\in C$, $y\not\in C$.
  Then, there exists $\alpha\in[0,1]$ such that $x+\alpha(y-x)\in\bd(C)$.
\end{prop}
\begin{proof}
  Define $\alpha:=\sup\{\hat{\alpha}\in[0,1]:x+\hat{\alpha}(y-x)\in C\}$ and fix $\epsilon>0$ arbitrarily.
  From the properties of the supremum, there exists $\beta>\alpha-\epsilon/\norm{y-x}$ such that $x+\beta(y-x)\in C$.
  Thus, we have
  \begin{align*}
    \norm{(x+\alpha(y-x))-(x+\beta(y-x))}=(\alpha-\beta)\norm{y-x}<\epsilon.
  \end{align*}
  Since $\epsilon>0$ was chosen arbitrarily, the above inequality implies that the point $(x+\alpha(y-x))$ is an adherent point of $C$.
  Furthermore, there exists $\gamma>\alpha+\epsilon/\norm{y-x}$ such that $x+\gamma(y-x)\not\in C$ due to the properties of the supremum.
  Hence, we also have
  \begin{align*}
    \norm{(x+\alpha(y-x))-(x+\gamma(y-x))}=(\gamma-\alpha)\norm{y-x}<\epsilon.
  \end{align*}
  Since $\epsilon>0$ was chosen arbitrarily, the above inequality implies that the point $(x+\alpha(y-x))$ is an adherent point of the complement of $C$.
  Therefore, $x+\alpha(y-x)$ belongs to the boundary of $C$.
  This completes the proof.
\end{proof}

\begin{prop}[{\cite[Corollary~2.15]{bc}}]\label{prop:normexpand}
  Let $x,y\in\Hsp$, and let $\alpha\in\real$.
  Then,
  \begin{align*}
    \norm{\alpha x+(1-\alpha)y}^2
    =\alpha\norm{x}^2+(1-\alpha)\norm{y}^2-\alpha(1-\alpha)\norm{x-y}^2
  \end{align*}
  holds.
\end{prop}

We start by proving the following lemma which leads us to the convergence rate of Algorithm~\ref{alg:main}.
\begin{lem}\label{lem:cvrate}
  Let $\{x_k\}\subset\Hsp$ be a sequence generated by Algorithm~\ref{alg:main}, and suppose that Assumptions~\ref{assum:basic} and \ref{assum:0} hold.
  Assume that the sequence $\{v_k\}$ is bounded.
  Then,
  \begin{align*}
    r_k
    \le\frac{\norm{x_i-x^\star}^2+\sum_{j=i}^k(1-\alpha_j)v_j^2}{2\sum_{j=i}^kv_j(1-\alpha_j)}
  \end{align*}
  for any $x^\star\in X^\star$, $k\in\nat$, and $i\in\{1,2,\ldots,k\}$.
\end{lem}
\begin{proof}
  Fix $x^\star\in X^\star$, $k\in\nat$, and $i\in\{1,2,\ldots,k\}$ arbitrarily.
  If $r_k$ is nonpositive, the statement obviously holds.
  Therefore, let us consider the case where $r_k$ is positive in the following.
  Fix $\delta\in(0,r_k)$ arbitrarily.
  The definition of $r_k$ and monotonicity of the sequence $\{r_k\}$ imply that $x^\star-\delta g_j$ belongs to the level set $\lev_{<f(x^\star_j)}f\subset \lev_{<f(x_j)}f$ for any $j=1,2,\ldots,k$.
  Therefore, $\ip{g_j,(x^\star-\delta g_j)-x_j}\le 0$ holds for all $j=1,2,\ldots,k$.
  Rearranging this inequality with the property $\norm{g_j}=1$, we have
  \begin{align*}
    \ip{g_j,x^\star-x_j}\le \delta
  \end{align*}
  for all $j=1,2,\ldots,k$.
  Here, the assumption $r_k>0$ implies that $f_\star<f(x_j)$ for all $j=1,2,\ldots,k$.
  Hence, all assumptions of Lemma~\ref{lem:fun} are satisfied and
  \begin{align*}
    \norm{x_{j+1}-j^\star}^2
    & \le \norm{x_j-x^\star}^2-2v_j(1-\alpha_j)\ip{g_j,x_j-x^\star}+(1-\alpha_j)v_j^2 \\
    & \le \norm{x_j-x^\star}^2-2\delta v_j(1-\alpha_j)+(1-\alpha_j)v_j^2
  \end{align*}
  is guaranteed by inequality~\eqref{eqn:2} for all $j=1,2,\ldots,k$.
  Summing the above inequalities from $j=i$ to $j=k$ yields
  \begin{align*}
    0
    & \le \norm{x_i-x^\star}^2-2\delta\sum_{j=i}^kv_j(1-\alpha_j)+\sum_{j=i}^k(1-\alpha_j)v_j^2.
  \end{align*}
  Transposing the term of $\delta$, we get
  \begin{align*}
    \delta
    & \le \frac{\norm{x_i-x^\star}^2+\sum_{j=i}^k(1-\alpha_j)v_j^2}{2\sum_{j=i}^kv_j(1-\alpha_j)}.
  \end{align*}
  The arbitrariness of $\delta\in(0,r_k)$ implies that
  \begin{align*}
    r_k
    & \le \frac{\norm{x_i-x^\star}^2+\sum_{j=i}^k(1-\alpha_j)v_j^2}{2\sum_{j=i}^kv_j(1-\alpha_j)}.
  \end{align*}
  This completes the proof.
\end{proof}

In general settings, we can use this lemma to analyze the convergence rate in terms of $r_k$.
The following two propositions present such an analysis for constant and diminishing step-size rules.
\begin{prop}[Convergence rate analysis for constant step-size rule]\label{prop:cvrate:const}
  Let $\{x_k\}\subset H$ be a sequence generated by Algorithm~\ref{alg:main}, and suppose that the assumptions in Theorem~\ref{thm:constant} hold.
  Then, there exists a number $k_0\in\nat$ such that
  \begin{align*}
      r_k
      & \le \frac{1}{1-\limsup_{j\to\infty}\alpha_j}\left(\frac{\norm{x_{k_0}-x^\star}^2}{(k-k_0+1)v}+\left(1-\frac{1}{2}\liminf_{j\to\infty}\alpha_j\right)v\right)
  \end{align*}
  holds for all $k\ge k_0$, in other words,
  \begin{align*}
      r_k
      & =\mathcal{O}(1/k+v).
  \end{align*}

  Furthermore,
  \begin{align*}
      r_k\le\frac{\|x_1-x^\star\|^2}{2(1-\alpha)kv}+\frac{1}{2}v
  \end{align*}
  holds for all $k\in\nat$ when the sequence $\{\alpha_k\}$ satisfies $\alpha_k=\alpha\in(0,1)$ for all $k\in\nat$.
\end{prop}

\begin{proof}
  Assumption~(A\ref{assum:alpha}) implies that there exists a number $k_0\in\nat$ such that
  \begin{align*}
      0 < \frac{1}{2}\left(1-\limsup_{k\to\infty}\alpha_k\right)
      \le 1-\alpha_k
      \le 1 - \frac{1}{2}\liminf_{k\to\infty}\alpha_k < 1
  \end{align*}
  for all $k\ge k_0$.
  Therefore, Lemma~\ref{lem:cvrate} with $i:=k_0$ leads to the finding that
  \begin{align*}
      r_k
      & \le \frac{\norm{x_{k_0}-x^\star}^2+\sum_{j=k_0}^k(1-\alpha_j)v_j^2}{2\sum_{j=k_0}^kv_j(1-\alpha_j)} \\
      & \le \frac{\norm{x_{k_0}-x^\star}^2+(1-\liminf_{j\to\infty}\alpha_j/2)(k-k_0+1)v^2}{(1-\limsup_{j\to\infty}\alpha_j)(k-k_0+1)v} \\
      & = \frac{1}{1-\limsup_{j\to\infty}\alpha_j}\left(\frac{\norm{x_{k_0}-x^\star}^2}{(k-k_0+1)v}+\left(1-\frac{1}{2}\liminf_{j\to\infty}\alpha_j\right)v\right)
  \end{align*}
  for all $k\ge k_0$.
  
  Furthermore, if $\{\alpha_k\}$ satisfies $\alpha_k=\alpha\in(0,1)$ for all $k\in\nat$, Lemma~\ref{lem:cvrate} with $i:=1$ leads to 
  \begin{align*}
      r_k
      &\le \frac{\|x_1-x^\star\|^2+\sum_{j=i}^k(1-\alpha_j)v_j^2}{2\sum_{j=i}^k v_j(1-\alpha_j)} \\
      &= \frac{\|x_1-x^\star\|^2+(1-\alpha)kv^2}{2(1-\alpha)kv} \\
      &= \frac{\|x_1-x^\star\|^2}{2(1-\alpha)kv}+\frac{1}{2}v
  \end{align*}
  for all $k\in\nat$.
  This completes the proof.
\end{proof}

\begin{prop}[Convergence rate analysis for diminishing step-size rule]\label{thm:crated}
  Let $\{x_k\}\subset H$ be a sequence generated by Algorithm~\ref{alg:main}, and suppose that the assumptions in Theorem~\ref{thm:dsr} hold.
  Let $c$ be a positive real number and assume that $v_k=c/k$ for all $k\in\nat$.
  Then, there exists a number $k_0\in\nat$ such that
  \begin{align*}
      r_k
      & \le \frac{\norm{x_{k_0}-x^\star}^2+2c^2(1-\liminf_{j\to\infty}\alpha_j/2)}{2c(1-\limsup_{j\to\infty}\alpha_j)(\log(k+1)-\log(k_0))}
  \end{align*}
  holds for all $k\ge k_0$, in other words,
  \begin{align*}
      r_k
      & =\mathcal{O}(1/\log(k+1)).
  \end{align*}
  
  In addition,
  \begin{align*}
      r_k
      \le \frac{\|x_1-x^\star\|^2+2c^2(1-\alpha)}{2c(1-\alpha)\log(k+1)}
  \end{align*}
  holds for all $k\in\nat$ when the sequence $\{\alpha_k\}$ satisfies $\alpha_k=\alpha\in(0,1)$ for all $k\in\nat$.
\end{prop}

\begin{proof}
  Assumption~(A\ref{assum:alpha}) implies that there exists a number $k_0\in\nat$ such that
  \begin{align*}
      0 < \frac{1}{2}\left(1-\limsup_{k\to\infty}\alpha_k\right)
      \le 1-\alpha_k
      \le 1 - \frac{1}{2}\liminf_{k\to\infty}\alpha_k < 1
  \end{align*}
  holds for all $k\ge k_0$.
  Therefore, Lemma~\ref{lem:cvrate} with $i:=k_0$ leads to the finding that
  \begin{align*}
      r_k
      & \le \frac{\norm{x_{k_0}-x^\star}^2+\sum_{j=k_0}^k(1-\alpha_j)v_j^2}{2\sum_{j=k_0}^kv_j(1-\alpha_j)} \\
      & \le \frac{\norm{x_{k_0}-x^\star}^2+c^2(1-\liminf_{j\to\infty}\alpha_j/2)\sum_{j=k_0}^k1/j^2}{2c(1-\limsup_{j\to\infty}\alpha_j)\sum_{j=k_0}^k 1/j}
  \end{align*}
  holds for all $k\ge k_0$.
  Using inequalities $\sum_{j=1}^\infty 1/j^2\le 2$, $\log(k+1)-\log(k_0)\le \sum_{j=k_0}^k 1/j$, we have
  \begin{align*}
      r_k
      & \le \frac{\norm{x_{k_0}-x^\star}^2+2c^2(1-\liminf_{j\to\infty}\alpha_j/2)}{2c(1-\limsup_{j\to\infty}\alpha_j)(\log(k+1)-\log(k_0))}
  \end{align*}
  for all $k\ge k_0$.

  In addition, if $\{\alpha_k\}$ satisfies $\alpha_k=\alpha\in(0,1)$ for all $k\in\nat$, Lemma~\ref{lem:cvrate} with $i:=1$ leads to 
  \begin{align*}
      r_k
      &\le \frac{\|x_1-x^\star\|^2+\sum_{j=i}^k(1-\alpha_j)v_j^2}{2\sum_{j=i}^k v_j(1-\alpha_j)} \\
      &= \frac{\|x_1-x^\star\|^2+c^2(1-\alpha)\sum_{j=i}^k 1/j^2}{2(1-\alpha)\sum_{j=i}^k 1/j} \\
      &\le \frac{\|x_1-x^\star\|^2+2c^2(1-\alpha)}{2(1-\alpha)\log(k+1)}
  \end{align*}
  for all $k\in\nat$.
  This completes the proof.
\end{proof}

The following theorem guarantees that the convergence rate with respect to the functional value can be bounded from above by one with respect to $r_k$ under certain assumptions.
This implies that, under the assumptions, we can deduce the convergence rate with respect to the functional value from the result of analyzing $r_k$.
We give this result after proving the following theorem.
\begin{thm}\label{thm:crate2}
  Suppose that the whole space $\Hsp$ is an $N$-dimensional Euclidean space $\real^N$.
  Let $\{x_k\}\subset\Hsp$ be a sequence generated by Algorithm~\ref{alg:main} and suppose that Assumptions~\ref{assum:basic} and \ref{assum:0} hold.
  Assume that $f_\star<f(x_k^\star)$ holds.
  Then,
  \begin{align*}
      f(x_k^\star)-f_\star
      & \le Lr_k^\beta
  \end{align*}
  holds for all $k\in\nat$.
\end{thm}
\begin{proof}
  Fix $k\in\nat$ and $x^\star\in X^\star$ arbitrarily.
  Furthermore, fix $j\in\nat$ arbitrarily.
  The complement of the slice $\lev_{<f(x_k^\star)}f$, i.e., the set $\{x\in\real^N:f(x_k^\star)\le f(x)\}$, is nonempty because $x_k^\star$ obviously belongs to it.
  Therefore, from the definition of $r_k$, there exists a point $v_j\in\real^N$ such that
  \begin{align*}
    v_j\not\in\lev_{<f(x_k^\star)}f\text{ and }\norm{x^\star-v_j}\le r_k+1/j
  \end{align*}
  hold.
  The assumption $f_\star<f(x_k^\star)$ implies that $x^\star\in\lev_{<f(x_k^\star)}f$, and the above discussion obtained the property $v_j\not\in\lev_{<f(x_k^\star)}f$.
  Therefore, by proposition~\ref{prop:isom}, a number $\alpha_j\in[0,1]$ exists such that $w_j:=x^\star+\alpha_j(v_j-x^\star)\in\bd(\lev_{<f(x_k^\star)}f)$.
  Let us consider the lower and upper bounds of the norm $\norm{x^\star-w_j}$.
  Here, the fact $w_j\not\in\lev_{<f(x_k^\star)}f$ from the continuity of the objective functional $f$ implies that $r_k\le\norm{x^\star-w_j}$ holds.
  The upper bound of the norm $\norm{x^\star-w_j}$ is thus
  \begin{align*}
      \norm{x^\star-w_j}
      &=\norm{x^\star-(x^\star+\alpha_j(v_j-x^\star))} \\
      &=\alpha_j\norm{x^\star-v_j} \\
      &\le\norm{x^\star-v_j} \\
      &\le r_k+1/j.
  \end{align*}
  The sequence $\{w_j\}$ is bounded, since $\norm{w_j}\le\norm{x^\star}+\norm{x^\star-w_j}\le\norm{x^\star}+r_k+1$ for all $j\in\nat$.
  Therefore, together with the closedness of the boundary $\bd(\lev_{<f(x_k^\star)}f)$, there exists a subsequence $\{w_{j_t}\}\subset\{w_j\}$ and a point $u_k\in\bd(\lev_{<f(x^\star_k)}f)$ such that $w_{j_t}$ converges to $u_k$.
  The continuity of $\norm{\cdot}$ leads us to the finding that
  \begin{align*}
      \norm{x^\star-u_k}
      & =\lim_{l\to\infty}\norm{x^\star-w_{j_l}} \\
      & =r_k
  \end{align*}
  due to the previous confirmation of the boundedness of the norm $\norm{x^\star-w_{j_t}}$; i.e., $r_k\le\norm{x^\star-w_{j_t}}\le r_k+1/j_t$ holds for any $t\in\nat$.
  Furthermore, the continuity of the objective functional $f$ ensures the coincidence $f(u_k)=f(x_k^\star)$.
  Now, Assumption~(A\ref{assum:holder}) guarantees that the functional $f$ satisfies the H\"older condition with degree $\beta>0$ at the point $x^\star$ on the set $\cl(\lev_{<f(x_k)}f)$.
  Therefore, we have
  \begin{align*}
      f(x_k^\star)-f_\star
      &= f(u_k)-f_\star \\
      &\le L\norm{x^\star-u_k}^\beta \\
      &= Lr_k^\beta.
  \end{align*}
  This completes the proof.
\end{proof}

This theorem directly induces the following corollary giving the convergence rate in terms of the objective functional when the diminishing step-size rule is adopted.
\begin{cor}\label{cor:connectholder}
  Suppose that the whole space $H$ is an $N$-dimensional Euclidean space $\real^N$ and the assumptions in Theorem~\ref{thm:dsr} hold.
  Let $c$ be a positive real number and assume that $v_k=c/k$ for all $k\in\nat$.
  Then, a number $k_0\in\nat$ exists such that
  \begin{align*}
    f(x_k^\star)-f_\star
    & \le L\left(\frac{\norm{x_{k_0}-x^\star}^2+2c^2(1-\liminf_{j\to\infty}\alpha_j/2)}{2c(1-\limsup_{j\to\infty}\alpha_j)(\log(k+1)-\log(k_0))}\right)^\beta
  \end{align*}
  holds for all $k\ge k_0$, in other words,
  \begin{align*}
    f(x_k^\star)-f_\star
    & =\mathcal{O}\left(\frac{1}{\left(\log(k+1)\right)^\beta}\right).
  \end{align*}

  In addition,
  \begin{align*}
    f(x_k^\star)-f_\star
    \le L\left(\frac{\|x_1-x^\star\|^2+2c^2(1-\alpha)}{2c(1-\alpha)\log(k+1)}\right)^\beta
  \end{align*}
  holds for all $k\in\nat$ when the sequence $\{\alpha_k\}$ satisfies $\alpha_k=\alpha\in(0,1)$ for all $k\in\nat$.
\end{cor}
\begin{proof}
  This is an immediate consequence of Proposition~\ref{thm:crated} and Theorem~\ref{thm:crate2}.
\end{proof}

We thus far have discussed the convergence rate of Algorithm~\ref{alg:main} in terms of the value of the objective functional.
The remaining part of this section shows another convergence rate analysis of Algorithm~\ref{alg:main}, namely in terms of the distance to the fixed point set.
The following theorem gives the convergence rate in terms of the distance to the fixed point set with respect to the averaged norm.
\begin{thm}\label{thm:cvrate:dsr}
  Suppose that the assumptions in Theorem~\ref{thm:dsr} hold.
  If $f_\star<f(x_k)$ for all $k\in\nat$, then
  \begin{align*}
      \frac{1}{k}\sum_{j=1}^k\norm{x_j-T(x_j)}
      & =\mathcal{O}(1/k).
  \end{align*}
\end{thm}
\begin{proof}
  Fix $k\in\nat$ arbitrarily.
  Using Proposition~\ref{prop:normexpand} and the fact that $P_D$ is a nonexpansive mapping and $x^\star$ is its fixed point, we have
  \begin{align*}
    \norm{x_{k+1}-x^\star}^2
    &=\norm{P_D(\alpha_kx_k+(1-\alpha_k)T(x_k-v_kg_k))-P_D(x^\star)}^2 \\
    &\le\norm{\alpha_k(x_k-x^\star)+(1-\alpha_k)(T(x_k-v_kg_k)-x^\star)}^2 \\
    &=\alpha_k\norm{x_k-x^\star}^2+(1-\alpha_k)\norm{T(x_k-v_kg_k)-x^\star}^2 \\
    &\quad-\alpha_k(1-\alpha_k)\norm{x_k-T(x_k-v_kg_k)}^2.
  \end{align*}
  Here, $x^\star$ is also a fixed point of the nonexpansive mapping $T$.
  Therefore, we can expand the second term of the above expression as follows:
  \begin{align*}
    \norm{T(x_k-v_kg_k)-x^\star}^2
    &=\norm{T(x_k-v_kg_k)-T(x^\star)}^2 \\
    &\le\norm{x_k-x^\star-v_kg_k}^2 \\
    &=\norm{x_k-x^\star}^2-2v_k\ip{x_k-x^\star,g_k}+v_k^2.
  \end{align*}
  Now, the assumption $f_\star<f(x_k)$ ensures that $-\ip{x_k-x^\star,g_k}\le 0$ holds because $g_k$ is a normal vector of the slice $\lev_{<f(x_k)}f$ at $x_k$.
  Hence, we have
  \begin{align*}
    \norm{T(x_k-v_kg_k)-x^\star}^2
    &=\norm{x_k-x^\star}^2+v_k^2.
  \end{align*}
  Overall, we have
  \begin{align*}
    \norm{x_{k+1}-x^\star}^2
    &\le\alpha_k\norm{x_k-x^\star}^2+(1-\alpha_k)(\norm{x_k-x^\star}^2+v_k^2) \\
    &\quad-\alpha_k(1-\alpha_k)\norm{x_k-T(x_k-v_kg_k)}^2 \\
    &=\norm{x_k-x^\star}^2-\alpha_k(1-\alpha_k)\norm{x_k-T(x_k-v_kg_k)}^2+(1-\alpha_k)v_k^2.
  \end{align*}
  Assumption~(A\ref{assum:alpha}) guarantees that a number $k_0\in\nat$ exists such that $\liminf_{j\to\infty}\alpha_j/2<\alpha_k<(1+\limsup_{j\to\infty}\alpha_j)/2$ holds for all $k\ge k_0$.
  Note that Assumption~(A\ref{assum:alpha}) also ensures that $0<\liminf_{j\to\infty}\alpha_k/2$ and $(1+\limsup_{j\to\infty}\alpha_j)/2<1$.
  Therefore, together with the above inequality, we have
  \begin{align*}
    \norm{x_{k+1}-x^\star}^2
    &\le\norm{x_k-x^\star}^2+\left(1-\frac{1}{2}\liminf_{j\to\infty}\alpha_j\right)v_k^2 \\
    &\quad-\frac{1}{4}\left(\liminf_{j\to\infty}\alpha_j\right)\left(1-\limsup_{j\to\infty}\alpha_j\right)\norm{x_k-T(x_k-v_kg_k)}^2 \\
    &\le\norm{x_{k_0}-x^\star}^2+\left(1-\frac{1}{2}\liminf_{j\to\infty}\alpha_j\right)\sum_{j=k_0}^kv_j^2 \\
    &\quad-\frac{1}{4}\left(\liminf_{j\to\infty}\alpha_j\right)\left(1-\limsup_{j\to\infty}\alpha_j\right)\sum_{j=k_0}^k\norm{x_j-T(x_j-v_jg_j)}^2
  \end{align*}
  if $k\ge k_0$.
  This inequality implies that
  \begin{align*}
    &\sum_{j=1}^k\norm{x_j-T(x_j-v_jg_j)}^2\\
    &\le\sum_{j=1}^{k_0-1}\norm{x_j-T(x_j-v_jg_j)}^2\\
    &\quad+4\left(\liminf_{j\to\infty}\alpha_j\right)^{-1}\left(1-\limsup_{j\to\infty}\alpha_j\right)^{-1}\left(\norm{x_{k_0}-x^\star}^2+\left(1-\frac{1}{2}\liminf_{j\to\infty}\alpha_j\right)\sum_{j=k_0}^kv_j^2\right)
  \end{align*}
  holds\footnote{If $k$ is less than $k_0$, we consider $\sum_{j=k_0}^kv_j^2=0$ here.}.
  Here, we assume that $\sum_{j=1}^\infty v_j^2$ converges.
  Therefore, the right side of the above inequality is bounded from above with respect to $k$.
  This implies that the sequence $\{\sum_{j=1}^k\norm{x_j-T(x_j-v_jg_j)}^2\}$ is also bounded from above.
  Let $M\in\real$ denote an upper bound of this sequence.

  Let us estimate the distance before and after applying the nonexpansive mapping to each approximation $x_k$.
  Using the parallelogram law and the nonexpansivity of $T$, we obtain
  \begin{align*}
    \norm{x_{k}-T(x_k)}^2
    &= \norm{x_{k}-T(x_k-v_kg_k)+T(x_k-v_kg_k)-T(x_k)}^2 \\
    &\le 2\norm{x_{k}-T(x_k-v_kg_k)}^2+2\norm{T(x_k-v_kg_k)-T(x_k)}^2 \\
    &\le 2\norm{x_{k}-T(x_k-v_kg_k)}^2+2v_k^2.
  \end{align*}
  Summing the above inequalities with respect to $k$ and dividing both sides by $k$, we get
  \begin{align*}
    \frac{1}{k}\sum_{j=1}^k\norm{x_j-T(x_j)}^2
    &\le \frac{1}{k}\left(2\sum_{j=1}^k\norm{x_j-T(x_j-v_jg_j)}^2+2\sum_{j=1}^k v_j^2\right) \\
    &\le \frac{1}{k}\left(2M+2\sum_{j=1}^\infty v_j^2\right).
  \end{align*}
  This completes the proof.
\end{proof}

We discussed two convergence analyses for Algorithm~\ref{alg:main} with the diminishing step-size rule.
By placing assumptions on the problem to be solved, we can also prove finite convergence.
The following proposition describes the requirements to obtain an optimal solution in a finite number of iterations.
\begin{prop}\label{prop:c1}
    Let $\{x_k\}\subset H$ be a sequence generated by Algorithm~\ref{alg:main}, and suppose that the assumptions in Theorem~\ref{thm:dsr} hold.
    Furthermore, assume that $X^\star$ has a nonempty interior and the sequence $\{x_k\}$ is contained inside $\Fix(T)$.
    Then, $x_k\in X^\star$ for some $k\in\nat$.
\end{prop}

\begin{proof}
  We will proceed by way of contradiction and suppose that the conclusion does not hold, that is, $f_\star<f(x_k)$ for all $k\in\nat$.
  Fix $x^\star\in X^\star$ arbitrarily.
  We deduce the following from inequality~\eqref{eqn:2}:
  \begin{align*}
      \norm{x_{k+1}-x^\star}^2
      &\le \norm{x_k-x^\star}^2-2v_k(1-\alpha_k)\ip{g_k,x_k-x^\star}+(1-\alpha_k)v_k^2 \\
      &\le \norm{x_1-x^\star}^2-2\sum_{j=1}^kv_j(1-\alpha_j)\ip{g_j,x_j-x^\star}+\sum_{j=1}^k(1-\alpha_j)v_j^2 \\
      &\le \norm{x_1-x^\star}^2-2\left(\min_{j=1,2,\ldots,k}\ip{g_j,x_j-x^\star}\right)\sum_{j=1}^k(1-\alpha_j)v_j+\sum_{j=1}^kv_j^2
  \end{align*}
  for all $k\in\nat$.
  The nonnegativity of the left side of the above inequality ensures that
  \begin{align}
      \min_{j=1,2,\ldots,k}\ip{g_j,x_j-x^\star}
      &\le \frac{\norm{x_1-x^\star}^2}{2\sum_{j=1}^k(1-\alpha_j)v_j}+\frac{\sum_{j=1}^kv_j^2}{2\sum_{j=1}^k(1-\alpha_j)v_j}  \label{eqn:c1e1}
  \end{align}
  for all $k\in\nat$.
  Now, there exists a positive real $\delta>0$ which satisfies $\delta\mathbf{B}+x^\star\subset X^\star$ because $X^\star$ has a nonempty interior.
  Therefore, $\delta\le\delta+\ip{g_k,x_k-(x^\star+\delta g_k)}=\ip{g_k,x_k-x^\star}$ holds for any $k\in\nat$.
  This property with inequality~\eqref{eqn:c1e1} implies
  \begin{align*}
      0<\delta
      &\le\min_{j=1,2,\ldots,k}\ip{g_j,x_j-x^\star} \\
      &\le \frac{\norm{x_1-x^\star}^2}{2\sum_{j=1}^k(1-\alpha_j)v_j}+\frac{\sum_{j=1}^kv_j^2}{2\sum_{j=1}^k(1-\alpha_j)v_j}
  \end{align*}
  for all $k\in\nat$.
  However, both terms of the right side of the above inequality converge to zero since $\limsup_{k\to\infty}\alpha_k<1$, $\sum_{j=1}^\infty v_j=\infty$ and $\sum_{j=1}^\infty v_j^2<\infty$.
  Therefore, we arrive at a contradiction.
  This completes the proof.
\end{proof}

The nonemptiness of the interior of minima appears in many interesting applications, such as surrogate relaxation of discrete programming problems \cite{dyer1980,yaohua2015}.
When we construct a nonexpansive mapping that transforms a given point into a fixed point of itself (an example of such a mapping is a metric projection, but notice that the assumption of this sentence is not limited to it) and give a fixed point of the mapping as the initial point to the algorithm, the generated sequence is contained within the fixed point set of the mapping due to its convexity.
Therefore, Proposition~\ref{prop:c1} can be applied to these situations.

\fi
\end{document}